\documentclass[10pt]{amsart}
\usepackage[utf8]{inputenc}
\usepackage[english]{babel}
\usepackage{amsthm}
\usepackage{amsmath}
\usepackage{mathrsfs}
\usepackage{amssymb}
\usepackage[a4paper, twoside=false, vmargin={2cm,3cm}, includehead]{geometry}
\usepackage{hyperref}
\usepackage{tikz}
\usepackage{graphicx}
\usepackage{mathtools}
\usepackage[capitalize]{cleveref}
\usepackage{accents}
\usepackage{yhmath}

\hypersetup{
    colorlinks,
    citecolor=black,
    filecolor=black,
    linkcolor=black,
    urlcolor=black
}

\newtheorem{theorem}{Theorem}[section]
\newtheorem{corollary}[theorem]{Corollary}
\newtheorem{lemma}[theorem]{Lemma}
\newtheorem{proposition}[theorem]{Proposition}

\newtheorem{conjecture}[theorem]{Conjecture}

\AtEndEnvironment{subproof}{\renewcommand{\qedsymbol}{$\blacksquare$}\qed}

\crefname{conjecture}{Conjecture}{Conjectures}
\Crefname{conjecture}{Conjecture}{Conjectures}

\theoremstyle{definition}
\newtheorem{definition}[theorem]{Definition}

\renewcommand\qedsymbol{$\blacksquare$}

\newcommand{\C}{\mathbb{C}}
\newcommand{\D}{\mathbb{D}}
\newcommand{\R}{\mathbb{R}}
\newcommand{\Z}{\mathbb{Z}}
\newcommand{\N}{\mathbb{N}}
\newcommand{\E}{\mathop{\mathbb{E}}}
\newcommand{\Es}{\mathop{\mathbb{E}^\times}_{k\in\N}}
\newcommand{\Ehfirst}{\mathop{\mathbb{E}}}
\newcommand{\Eh}{\Ehfirst_{\Delta_N}}
\newcommand{\T}{\mathbb{T}}
\newcommand{\M}{\mathcal{M}}
\renewcommand{\S}{\mathbb{S}}
\renewcommand{\P}{\mathbb{P}}
\renewcommand{\d}{\, \mathtt{d}}
\newcommand{\dbtilde}[1]{\accentset{\approx}{#1}}

\title{Pythagorean triples in level sets of completely multiplicative functions}
\author{Guilherme Azevedo}
\address[Guilherme Azevedo]{Universidade NOVA de Lisboa, Caparica, Portugal} \email{g.azevedo@campus.fct.unl.pt}
\author{Joel Moreira}
\address[Joel Moreira]{University of Warwick, Coventry, UK}		 \email{joel.moreira@warwick.ac.uk}
\date{}

\begin{document}

\begin{abstract}
We show that given completely multiplicative functions $f_1,\dots,f_d$ taking values in the unit circle, there exist Pythagorean triples (i.e., integer solutions to $x^2+y^2=z^2$) with $f_i(x),f_i(y),f_i(z)$ all arbitrarily close to $1$ for all $i$. 
This is a new special case of the conjecture that any finite colouring of $\N$ has a monochromatic Pythagorean triple.
Our proof combines vanishing averages for aperiodic functions with concentration estimates for pretentious functions. 
A similar proof is applied to obtain the analogous statement for more general equations of the form $ax^2+by^2=cz^2$ whenever $a,b,c$ are perfect squares satisfying the Rado's condition.
\end{abstract}

\maketitle
\tableofcontents

\section{Introduction}

\subsection{Statement of main results}

In \cite{schur1916kongruenz} Schur showed that the equation $x+y=z$ is \emph{partition regular}, in the sense that any finite colouring of the natural numbers yields a solution $(x,y,z)$ to the equation with all variables in the same colour.
In \cite{rado1933studien}, Rado extended Schur's result by giving necessary and sufficient conditions for a system of linear equations to be partition regular. In particular, it follows from Rado's theorem that for any coefficients $a,b,c\in\N:=\{1,2,3,\dots\}$, the equation $ax+by=cz$ is partition regular if and only if $a=c, b=c,$ or $a+b=c$.

More recently, attention has shifted to partition regularity of polynomial equations, but a question of Erd\H{o}s and Graham \cite{Graham07} regarding Pythagorean triples -- integer solutions to the equation $x^2+y^2=z^2$ -- remains unsolved. 
We believe the answer to be positive and so phrase this problem here as a conjecture.

\begin{conjecture}\label{conj_pythtriples}
    Any finite colouring of $\N$ has a monochromatic Pythagorean triple.
\end{conjecture}

In \cite{Heule_2016}, a computer-aided search established that any 2-colouring of the natural numbers has a monochromatic Pythagorean triple, but already for three colours the problem remains unsolved. 
The related equation $x_1^2+x_2^2+x_3^2+x_4^2=x_5^2$ was shown to be partition regular in \cite{chow2021rado} using a transference principle approach, but this method seems unable to handle the Pythagorean equation without some significant new ideas.
In a different direction, in \cite{frantzikinakis2024partitionregularitypythagoreanpairs} it was shown that given any finite colouring of $\N$ there exists a Pythagorean triple $(x,y,z)$ with $x$ and $y$ in the same colour; and a (potentially different) Pythagorean triple $(x,y,z)$ with $x$ and $z$ in the same colour.

There is an interesting connection between partition regularity of homogeneous equations, such as $x^2+y^2=z^2$, and unimodular completely multiplicative functions, i.e., functions $f:\N\to\S^1:=\{z\in\C:|z|=1\}$ satisfying $f(nm)=f(n)f(m)$ for every $m,n\in\N$.
We denote the class of all such functions by $\M$.
For example, for those systems of homogeneous \emph{linear} equations that are not partition regular, the natural examples of a colouring without a monochromatic solution can be realized as the composition of a completely multiplicative function $f:\N\to\S^1$ with a partition of $\S^1$ into finitely many arcs.
This leads us to consider the special case of \cref{conj_pythtriples} using such colourings.
In \cite[Theorem 1.5]{frantzikinakis2024partitionregularitypythagoreanpairs}, it was shown that if $f\in\M$ takes only finitely many values, then the finite colouring of $\N$ it induces does contain monochromatic Pythagorean triples.
In \cite[Corollary 1.2]{frantzikinakis2025recurrencepretentioussystemsgeneralized}, it is shown that for given $f_1,\dots,f_d\in\M$ and an arc $I\subset\S^1$ containing $1$, there exists a Pythagorean triple $(x,y,z)$ with $f_i(x),f_i(y),f_i(z)\in I$ for all $i=1,\dots,d$; provided the functions $f_i$ are all \emph{pretentious} (see \cref{def:pretentiousFunction} below).
The main goal of this paper is to remove the condition that the functions are pretentious.

\begin{theorem}\label{thm_mainpythagorean}
    Let $f_1,\dots,f_d\in\M$ and let $I\subset\S^1$ be an open set containing $1$. 
    Then there exists a solution to $x^2+y^2=z^2$ with $f_i(x),f_i(y),f_i(z)\in I$ for all $i=1,\dots,d$.
\end{theorem}

Keeping in mind the special case of Rado's theorem concerning equations of the form $ax+by=cz$, a more general form of \cref{conj_pythtriples} was formulated in \cite[Conjecture 1]{frantzikinakis2025recurrencepretentioussystemsgeneralized}.

\begin{conjecture}
    Given coefficients $a,b,c\in\N$, the equation
    \begin{equation}\label{eq:genPyth}
        ax^2+by^2=cz^2
    \end{equation}
    is partition regular if and only if $c\in\{a,b,a+b\}$.
\end{conjecture}

It follows directly from Rado's theorem that the condition $c\in\{a,b,a+b\}$ is necessary, but whether it is sufficient is still open.
In fact, it is unknown whether there exist \emph{any} $a,b,c\in\N$ for which \eqref{eq:genPyth} is partition regular.
Our proof of \cref{thm_mainpythagorean} can be adapted to cover more general equations.

\begin{theorem}\label{thm_maingenpyth}
    Let $a,b,c\in\N$ be perfect squares and suppose that either $a=c$ or $a+b=c$. 
    Let $f_1,\dots,f_d\in\M$ and let $I\subset\S^1$ be an open set containing $1$. 
    Then there exists a solution to $ax^2+by^2=cz^2$ with $f_i(x),f_i(y),f_i(z)\in I$ for all $i=1,\dots,d$.
\end{theorem}
Our proof combines Gowers uniformity of aperiodic functions $f\in\M$, obtained in \cite{frantzikinakis2016higherorderfourieranalysis}, with concentration estimates for pretentious multiplicative functions obtained in \cite{Klurman_2021} and \cite{frantzikinakis2024partitionregularitygeneralizedpythagorean}.
The reliance on these tools forces us to consider only equations whose parametric solutions (see \cref{sec_parametrizations} below) only use one irreducible polynomial.
This is the case only when all coefficients $a,b,c$ are perfect squares.

\subsection{Ergodic interpretation}
In this section we explore an equivalent form of the \cref{conj_pythtriples} involving measure preserving actions of the semigroup $(\N,\times)$ of natural numbers under multiplication. 
This interpretation will not be used again in the proof (or in fact anywhere in the paper outside this section).

Given a probability space $(X,\mu)$ we say that a map $T:X\to X$ \emph{preserves the measure} if $\mu(T^{-1}A)=\mu(A)$ for every measurable set $A\subset X$.
If $T=(T_n)_{n\in\N}$ is an action of $(\N,\times)$ on $X$ and each $T_n$ preserves the measure we say that $(X,\mu,T)$ is a \emph{multiplicative measure preserving system}, or just \emph{multiplicative system}.

Furstenberg's correspondence principle, when applied to \cref{conj_pythtriples} yields the following equivalent form.
\begin{conjecture}[{\cite[Conjecture 2]{frantzikinakis2025recurrencepretentioussystemsgeneralized}}]\label{conj_mainergodic}
    Let $(X,\mu,T)$ be a multiplicative system and let $A\subset X$.
    Suppose that $\mu(A\cup T_1^{-1}A\cup\cdots T_r^{-1}A)=1$ for some $r\in\N$.
    Then there exists a solution $(x,y,z)$ to $x^2+y^2=z^2$ such that
    $$\mu(T_x^{-1}A\cap T_y^{-1}A\cap T_z^{-1}A)>0.$$
\end{conjecture}
We stress that the condition $\mu(A\cup T_1^{-1}A\cup\cdots T_r^{-1}A)=1$ cannot be replaced by the weaker condition $\mu(A)>0$ (cf. second remark after Theorem 1.1 in \cite{frantzikinakis2024partitionregularitypythagoreanpairs}).
\cref{conj_mainergodic} is a double recurrence statement in ergodic theory, faintly resembling the ergodic Roth theorem.
The first step in Furstenberg's proof of Roth's theorem is to establish the special case of Kronecker systems.
We recall here the definition for multiplicative systems.

\begin{definition}\label{def_kronecker}
    Let $(X,\cdot)$ be a compact abelian group with Haar measure $\mu$.
    Let $\alpha:\N\to X$ be a multiplicative homomorphism, i.e., satisfy $\alpha(nm)=\alpha(n)\cdot\alpha(m)$.
    Then the maps $T_n:X\to X$ given by $T_n(x)=x\alpha(n)$ form a multiplicative measure preserving action $T=(T_n)_{n\in\N}$ and $(X,\mu,T)$ is a multiplicative system.
    A multiplicative system is called a \emph{Kronecker system} if it is (isomorphic to a system) of this form.
\end{definition}

An equivalent form of \cref{thm_mainpythagorean} is the special case of \cref{conj_mainergodic} for Kronecker systems:

\begin{theorem}\label{thm_Kroneckerfactor}
    Let $(X,\mu,T)$ be a Kronecker multiplicative system and let $A\subset X$ with $\mu(A)>0$.
    Then there exists a solution $(x,y,z)$ to $x^2+y^2=z^2$ such that
    $$\mu(A\cap T_x^{-1}A\cap T_y^{-1}A\cap T_z^{-1}A)>0.$$
\end{theorem}
Note that in a Kronecker system the condition $\mu(A)>0$
is equivalent to the condition $\mu(A\cup T_1^{-1}A\cup\cdots T_r^{-1}A)=1$ for some $r\in\N$.
A related result obtained in \cite{frantzikinakis2025recurrencepretentioussystemsgeneralized} derives the same conclusion under the assumption that the spectral measure of the system is supported on the set of pretentious functions. 
By contrast, \cref{thm_Kroneckerfactor} allows for aperiodic functions to be part of the spectrum, but it requires the spectral measure to be countably supported.

\begin{proof}
Let $\epsilon>0$ be sufficiently small in terms of $\mu(A)$.
    Since $X$ is a compact abelian group, the characters form an orthonormal basis for $L^2$ and hence one can find characters $\chi_1,\dots,\chi_d\in\hat X$ and coefficients $c_1,\dots,c_d\in\C$ such that 
    $$\left\|1_A-\sum_{i=1}^dc_i\chi_i\right\|_{L^2}<\epsilon.$$
    Let $\alpha:\N\to X$ be the homomorphism from \cref{def_kronecker} and, for each $i=1,\dots,d$, let $f_i=\chi_i\circ\alpha$. 
    Note that $f_i\in\M$ and $\chi_i\circ T_n=f_i(n)\chi_i$.

    Let $I\subset\S^1$ be the arc around $1$ with radius $\epsilon/\sum|c_i|$.
    Using \cref{thm_mainpythagorean}, find a Pythagorean triple $(x,y,z)$ such that $f_i(x),f_i(y),f_i(z)\in I$ for all $i=1,\dots,d$. 
    It follows that
    $$\|1_A\circ T_x-1_A\|_{L^2}\leq 2\epsilon+\sum_{i=1}^d|c_i|\cdot\|\chi_i\circ T_x-\chi_i\|_{L^2}<3\epsilon.$$
    Similarly for $y$ and $z$, which leads to the conclusion
    \begin{equation*}
        \mu(A\cap T_x^{-1}A\cap T_y^{-1}A\cap T_z^{-1}A) = \int_X1_A\cdot 1_A\circ T_x\cdot 1_A\circ T_y\cdot 1_A\circ T_z\d\mu
        \geq \int_X1_A\d\mu-9\epsilon>0
    \end{equation*}
    whenever $\epsilon<\mu(A)/9$.
    \end{proof}
To parallel the ergodic proof of Roth's theorem, one might attempt to ``lift'' \cref{thm_Kroneckerfactor} from the Kronecker factor of a given system to the entire system.
However, even for systems where the Kronecker factor is trivial (equivalently, when the system is weak-mixing) \cref{conj_mainergodic} is still open.

\subsection{Parametrizations}
\label{sec_parametrizations}

To establish our main theorems we consider parametrized solutions to \eqref{eq:genPyth} which are given by polynomials in two variables.
For the rest of the paper we let $a,b,c\in\N$ be perfect squares.

When $a=c$ (and we will omit the case $b=c$ because it is symmetric), solutions to \eqref{eq:genPyth} are given, for each choice of $m,n\in\N$, by $(P_x,P_y,P_z)$, where
\begin{equation}\label{eq:a=c=1_param}
    P_x=P_x(m,n)=m^2-bn^2,\quad P_y=P_y(m,n)=2\sqrt{a}mn,\quad P_z=P_z(m,n)=m^2+bn^2.
\end{equation}
Note that $P_z$ is irreducible, whereas $P_x$ and $P_y$ are both reducible (if $b=\beta^2$ then $P_x=(m+\beta n)(m-\beta n)$).
We have that $P_x,P_y,P_z$ are all positive whenever $m>n\sqrt{b}$.

When $a+b=c$, solutions to \eqref{eq:genPyth} are given, for each choice of $m,n\in\N$, by $(P_x,P_y,P_z)$, where
\begin{equation}\label{eq:a+b=c_param}
    \begin{gathered}
    P_x=P_x(m,n)=m^2-2bmn-abn^2,\qquad\qquad\qquad P_y=P_y(m,n)=m^2+2amn-abn^2,\\
    P_z=P_z(m,n)=m^2+abn^2.
    \end{gathered}
\end{equation}
Again note that $P_z$ is irreducible, but both $P_x$ and $P_y$ are reducible, indeed if $\alpha=\sqrt{bc}$ and $\gamma=\sqrt{ac}$ then
$$P_x= (m-bn+\alpha n)(m-bn-\alpha n)\quad \text{ and }\quad P_y= (m+an+\gamma n)(m+an-\gamma n).$$
We have that $P_x,P_y,P_z$ are all positive whenever $m>(a+2b)n$.

To find a solution to \eqref{eq:genPyth}, we will consider averages over $(m,n)\in\N^2$.
To ensure that $x,y,z$ are positive, throughout the paper we restrict our averages to the region $m>(a+2b)n$.

\subsection{Averaging schemes}
We use the standard notation
$$\E_{n\in A}f(n):= \frac{1}{|A|}\sum_{n\in A}f(n),$$
whenever $A$ is a finite set and $f:A\to\C$.
In order to restrict attention to the region $m>(a+2b)n$ so that all of $P_x,P_y,P_z$ are positive, for each $N\in \N$ we define 
$$\Delta_N:=\big\{(m,n)\in[N]^2:\,m>(a+2b)n\big\}$$ 
and we will use the notation
$$\Eh f(m,n):=\E_{(m,n)\in\Delta_N}f(m,n)$$
for sufficiently large $N\in\N$ and $f:\N^2\to\C$.

A \emph{multiplicative F{\o}lner sequence} in $\N$ is a sequence $\Phi=(\Phi_K)_{K=1}^\infty$ of finite subsets of $\N$ asymptotically invariant under dilation in the sense that
$$\forall x\in\N,\quad \lim_{K\to\infty}\frac{|\Phi_K\cap (x\cdot\Phi_K)|}{|\Phi_K|}=1.$$
The \emph{upper density} of a set $A\subset\N$ with respect to a multiplicative F{\o}lner sequence $\Phi$ is the quantity
    $$\bar{d}_\Phi(A):= \limsup_{K\to\infty}\frac{|\Phi_K\cap A|}{|\Phi_K|},$$
    and we say that $A$ has \emph{positive multiplicative density} (or that it is \emph{multiplicatively large}) if $\bar{d}_\Phi(A)>0$ for some multiplicative F{\o}lner sequence $\Phi$.

Throughout the paper we make use of the following multiplicative F{\o}lner sequence $\Phi=(\Phi_K)_{K\in\N}$. 
\begin{definition}\label{def:Phi_K}
    For each $K\in\N$, let
    \begin{equation*}
        \Phi_K := \Bigg\{\prod_{p\leq K}p^{a_p}: K<a_p\leq 2K\Bigg\}.
    \end{equation*}
\end{definition}
It is easy to check that $\Phi$ is indeed a multiplicative F{\o}lner sequence.
It has the useful additional property that for any fixed $Q\in\N$, if $K>Q$ then any element of $\Phi_K$ is a multiple of $Q$; this becomes useful in \cref{sec_concentration} to describe the concentration estimate results we need.

In the rest of the paper we use the notation 
$$\Es f(k):=\lim_{K\to\infty}\E_{k\in\Psi_K}f(k),$$ 
where $(\Psi_K)_{K\in\N}$ is a subsequence of the sequence $(\Phi_K)_{K\in\N}$ described in \cref{def:Phi_K} for which all the limits involved exist.\footnote{One may describe this more formally as $\displaystyle \Es f(k):=\lim_{K\to p}\E_{k\in\Phi_K}f(k)$ for some non-principal ultrafilter $p\in\beta\N\setminus\N$.}

\subsection{Jointly Pythagorean functions}

As a special case of the parametrization \eqref{eq:a=c=1_param}, a solution for the Pythagorean equation $x^2+y^2=z^2$ is given by
\begin{equation*}
    x=k(m^2-n^2),\quad y=2kmn,\quad z=k(m^2+n^2)
\end{equation*}
for any choice of $m,n,k\in\N$ with $m>n$.
We establish \cref{thm_mainpythagorean} by proving a stronger statement, namely that if $m,n,k\in\N$ are chosen at random in an appropriate sense, then with positive probability all $f_i(x),f_i(y),f_i(z)$ lie inside the given interval.
To make this statement more precise we introduce the following notion.

\begin{definition}[Jointly Pythagorean]
    Let $d\in\N$ and $f_1,\dots,f_d\in\mathcal{M}$. 
    We say that $f_1,\dots,f_d$ are \emph{jointly Pythagorean} if for every continuous $\psi:\S^{1}\to[0,1]$ with $\psi(1)>0$,
    \begin{equation*}
        \limsup_{N\to\infty}\E_{m,n\in[N]}\mathbf{1}_{m>n}\Es
        \prod_{j=1}^d \psi\Big(f_j\big(k(m^2-n^2)\big)\Big)\psi\Big(f_j\big(2kmn\big)\Big)\psi\Big(f_j\big(k(m^2+n^2)\big)\Big)>0.
    \end{equation*}
\end{definition}
\cref{thm_mainpythagorean} follows from the following stronger averaging result.

\begin{theorem}\label{thm:main}
    Any finite collection $f_1,\dots,f_d\in\M$ is jointly Pythagorean.
\end{theorem}

Similarly, to establish \cref{thm_maingenpyth} we prove a stronger averaging theorem. 

\begin{definition}
\label{def:JP}
    Let $a,b,c\in\N$ be perfect squares satisfying either $a=c$ or $a+b=c$.
    Let $f_1,\dots,f_d\in\mathcal{M}$. 
    We say that $f_1,\dots,f_d$ are \emph{jointly Pythagorean (\emph{JP}) with respect to $a,b,c$} if for every continuous $\psi:\S^{1}\to[0,1]$ with $\psi(1)>0$,
    \begin{equation}\label{eq_jpdef}
        \limsup_{N\to\infty}\Eh~\Es~
        \prod_{j=1}^d \psi\Big(f_j\big(kP_x(m,n)\big)\Big)\psi\Big(f_j\big(kP_y(m,n)\big)\Big)\psi\Big(f_j\big(kP_z(m,n)\big)\Big)>0,
    \end{equation}
    where $P_x,P_y,P_z\in\Z[m,n]$ are described by \eqref{eq:a=c=1_param} if $a=c$, or by \eqref{eq:a+b=c_param} if $a+b=c$.
\end{definition}
Our main theorem is the following. Note that in particular it implies both Theorems \ref{thm:main} and \ref{thm_maingenpyth}.
\begin{theorem}\label{thm_jpforabc}
Let $a,b,c\in\N$ be perfect squares satisfying either $a=c$ or $a+b=c$.
    Then any finite collection $f_1,\dots,f_d\in\M$ is JP w.r.t. $a,b,c$.
\end{theorem}

The case $a=c$ is proven in \cref{thm:general_a=c}, and the case $a+b=c$ is proven in \cref{thm:general_a+b=c}.

\subsection{Main novelties}

Our departure point is the proof of \cite[Theorem 1.5]{frantzikinakis2024partitionregularitypythagoreanpairs}, which deals with the special case of \cref{thm_mainpythagorean} where $d=1$ and the function takes only finitely many values.
Besides the immediate advantage that the Fourier analysis of finitely valued multiplicative functions is technically simpler, that special case has also the advantage that no archimedean characters are involved, and the only functions that ``pretend to be 1'' differ from the constant function 1 at only finitely many primes.
In the general case, a variety of new complications arise.

While the main tools we use to prove Theorems \ref{thm_mainpythagorean} and \ref{thm_maingenpyth} are essentially the same as those used in \cite{frantzikinakis2024partitionregularitypythagoreanpairs} and \cite{frantzikinakis2024partitionregularitygeneralizedpythagorean}, the new difficulties require a rather delicate technical analysis, especially regarding archimedean characters, which we handle with the help of \cref{cor:weightforpolynomials}. 
Another novelty is \cref{lem:structuralResult}, which allows us to handle the general case involving both aperiodic and pretentious functions at once, by reducing the general case to a more manageable special case.
Another difficulty is presented by functions $f\in\M$ that, say, pretend to be $1$, but for which the series $\sum_p|f(p)-1|/p$ diverges. Moreover, for more than one such function, the divergence may be witnessed at different scales. To deal with these functions in a robust manner, we use \cref{lem:sequencesandseries}.

\subsection*{Acknowledgments}
The first author was funded through the FCT – Fundação para a Ciência e a Tecnologia, I.P., under the scope of the projects UID/00297/2025 (\url{https://doi.org/10.54499/UID/00297/2025}) and UID/PRR/00297/2025 (\url{https://doi.org/10.54499/UID/PRR/00297/2025}) (Center for Mathematics and Applications – NOVA Math), and the studentship with reference 2023.00487.BD (\url{https://doi.org/10.54499/2023.00487.BD}).
The second author was supported by the EPSRC Frontier Research Guarantee grant EP/Y014030/1.

\section{Background}

\subsection{Multiplicative functions}

Following Granville and Soundararajan \cite{granville2014multiplicative}, we recall here the basic notions and results from pretentious number theory that are used throughout this paper.
Recall that $\M$ denotes the set of all completely multiplicative functions $f:\N\to\S^1$.

\begin{definition}[Dirichlet character]
    Given $q\in\N$ and a multiplicative character $f:(\Z/q\Z)^\times\to\S^1$, define the associated \textit{Dirichlet character}\footnote{In the literature these are usually called ``modified Dirichlet characters'' as they differ from the common Dirichlet characters only on those finitely many primes that divide $q$. As classical Dirichlet characters are not in $\M$ (they take the value $0$) we chose to drop the word ``modified'' in this article.} as the function $\chi\in\M$ with values on primes $p$ given by
    $$\chi(p):=\begin{cases}
        1,&\text{if }p|q,\\
        f(p\bmod q),&\text{otherwise.}
    \end{cases}$$
\end{definition}

\begin{definition}[Pretentious distance]
    For $f,g\in\M$ and $x,y\in\R^+$ with $x<y$ define
    $$\mathbb{D}^2(f,g;x,y):= \sum_{x<p\leq y}\frac{1}{p}\Big|f(p)-g(p)\Big|^2.$$
    We also let
    $$\D^2(f,g):=\lim_{y\to\infty}\mathbb{D}^2(f,g;1,y)=
    \sum_{p\in\P}\frac{1}{p}\Big|f(p)-g(p)\Big|^2$$
    It can be shown that $\D$ is a metric. 
    Moreover, for any $f_1,f_2,g_1,g_2\in\M$ we have (see \cite[Lemma 3.1]{granville2005largecharactersumspretentious})
    $$\mathbb{D}(f_1f_2,g_1g_2)\leq \mathbb{D}(f_1,g_1)+\mathbb{D}(f_2,g_2).$$
\end{definition}

\begin{definition}[Pretentious function]\label{def:pretentiousFunction}
    If $f,g\in\M$ we say that \textit{$f$ pretends to be $g$}, and write $f\sim g$, if $\D(f,g)<\infty$. 
    If $f\sim\chi\cdot n^{it}$ for some $t\in\R$ and Dirichlet character $\chi$, we say that $f$ is \textit{pretentious}.
    In this case both $\chi$ and $t$ are uniquely determined.
\end{definition}

Note that the set of all pretentious functions form a subgroup of $\M$.
\begin{definition}[Aperiodic function]
    A function $f\in\M$ is said \textit{aperiodic} if for every $a,b\in\N$
    $$\lim_{N\to\infty}\frac{1}{N}\sum_{n=1}^Nf(an+b)=0.$$
\end{definition}

The next lemma, which is a standard result from pretentious number theory, makes it possible to split our arguments in two separate cases (see \cite[Corollary 1]{daboussi1982multiplicative}).

\begin{lemma}[Daboussi-Delange]
    Let $f\in\M$. Then either $f$ is pretentious or $f$ is aperiodic. The only $f\in\M$ that is both pretentious and aperiodic is the constant function $f\equiv1$.
\end{lemma}

It follows directly from this lemma that if $f$ is aperiodic (and not the constant $1$ function) and $g$ is pretentious, then the product $fg$ is aperiodic; we will often use this fact tacitly.

We also make use of the following fact.

\begin{lemma}[{\cite[Lemma 3.2]{frantzikinakis2024partitionregularitypythagoreanpairs}}]\label{lem:folnervanish}
    Let $\Phi$ be a multiplicative F{\o}lner sequence and $f\in\M$ be non-trivial. Then
    $$\lim_{K\to\infty}\E_{n\in\Phi_K}f(n)=0.$$
\end{lemma}

\subsection{Concentration estimates}\label{sec_concentration}

We make essential use of concentration estimates for pretentious functions, which allow one to control them on certain arithmetic progressions.
Given $Q\in\N$ and $P\in\Z[m,n]$, we denote $P(Qm+1,Qn)$ by $P^Q(m,n)$.

Throughout this section we make use of the F{\o}lner sequence $\Phi$ defined in \cref{def:Phi_K}.
For the reader's convenience, we reproduce here the full statement of \cite[Proposition 4.8]{frantzikinakis2024partitionregularitygeneralizedpythagorean}.

\begin{proposition}\label{prop:irrConcentration}
    Let $P\in\Z[m,n]$ be an irreducible binary quadratic form, $A,K,N\in\N$ and $f:\Z\to\C$ be an even multiplicative function with $\|f\|_\infty\leq1$, $t\in\R$, $\chi$ be a Dirichlet character with period $q$, $Q=\prod_{p\leq K}p^{a_p}$ for some $a_p\in\N$. 
    Let also $a_0,b_0\in\Z$ be such that $-A\leq a_0,b_0\leq A$, $c_0:= \gcd(P(a_0,b_0),Q)$, and suppose that $q$ and $\prod_{p\leq K}p$ both divide $Q/c_0$. Then there exist constants $C_{A,P,Q}$ such that if $K$ is sufficiently large, depending only on $P$, and $N$ is sufficiently large, depending only on $A,P,Q$, we have
    \begin{equation*}
        \begin{gathered}
        \E_{m,n\in[N]}\Big|f\big(P_{c_0}(Qm+a_0, Qn+b_0)\big) - \chi\big(P_{c_0}(a_0,b_0)\big)\cdot \big(P_{c_0}(Qm,Qn)\big)^{it}\cdot \exp\big(G_P(f;K,N)\big)\Big|\\
        \ll_P \big(\mathbb{D}_P+\mathbb{D}^2_P\big)\big(f,\chi\cdot n^{it}; K,\sqrt{N}\big) + C_{A,P,Q}\cdot \mathbb{D}_P\big(f, \chi\cdot n^{it}; \sqrt{N}, C_{A,P,Q}N^2\big) + K^{-1/2},
        \end{gathered}
    \end{equation*}
    where $P_{c_0}:= P/c_0$ and
    \begin{equation}\label{eq:G(f;K,N)def}
        G_P\big(f;K,N\big):= \sum_{K<p\leq N} \frac{\omega_P(p)}{p}\Big(f(p)\cdot\overline{\chi(p)}\cdot p^{-it}-1\Big),
    \end{equation}
    \begin{equation*}
        \mathbb{D}^2_P\big(f,g; x,y\big):= \sum_{x<p\leq y} \frac{\omega_P(p)}{p}\Big(1-\mathfrak{R}\big(f(p)\cdot\overline{g(p)}\big)\Big),
    \end{equation*}
    where $\omega_P(p):=\{n\in\Z_p: P(n,1)=0\bmod{p}\}.$
\end{proposition}

We only make use of this proposition in the form of the following corollary.
\begin{corollary}\label{cor:irrConcentration}
    Let $P\in\Z[m,n]$ be an irreducible binary quadratic form satisfying $P(1,0)=1$, and $f\in\M$ such that $f\sim \chi\cdot n^{it}$ for some $t\in\R$ and Dirichlet character $\chi$. 
    Let $\epsilon>0$. 
    For large enough $K$, any $Q\in\Phi_K$ satisfies
    \begin{equation}\label{eq_cor:Prop4.8}
        \limsup_{N\to\infty} \Eh \Big|f\big(P(Qm+1,Qn)\big)-\big(P(Qm,Qn)\big)^{it}\cdot\exp\big(G_P(f;K,N)\big)\Big| \leq \epsilon
    \end{equation}
\end{corollary}

\begin{proof}
    We apply \cref{prop:irrConcentration} with $A=a_0=1$, $b_0=0$ and $Q\in\Phi_K$ arbitrary, noting that for sufficiently large $K$ (depending only on $P$), any $Q\in\Phi_K$ satisfies the necessary hypothesis. In particular, $c_0=P(1,0)=1$, which implies that $P_{c_0}=P$ and $\chi(P_{c_0}(1,0))=1$

    Let $C=C_{A,P,Q}$ given by \cref{prop:irrConcentration}.
    Since $f\sim \chi\cdot n^{it}$,
    $$\mathbb{D}^2(f,\chi\cdot n^{it}):= \sum_{p\in\mathbb{P}}\frac{1}{p}\Big(1-\mathfrak{R}\Big(f(p)\cdot\overline{\chi(p)}p^{-it}\Big)\Big)< \infty.$$
    The weight $\omega_P(p)$ is at most $2$ by \cite[Lemma 4.3]{frantzikinakis2024partitionregularitygeneralizedpythagorean}. Hence $\mathbb{D}^2_P(f,\chi\cdot n^{it})<\infty$, implying that 
    $$\lim_{K\to\infty}\big(\mathbb{D}_P+\mathbb{D}^2_P\big)\big(f,\chi\cdot n^{it}; K,\infty\big) + K^{-1/2}= \lim_{N\to\infty}C\mathbb{D}_P\big(f,\chi\cdot n^{it}; \sqrt{N}, CN^2\big)=0.$$
    So, for any sufficiently large $K$ we have 
    \begin{equation*}
        \limsup_{N\to\infty} \E_{m,n\in[N]} \Big|f\big(P(Qm+1,Qn)\big)-\big(P(Qm,Qn)\big)^{it}\cdot\exp\big(G_P(f;K,N)\big)\Big| \leq \epsilon.
    \end{equation*}
    Since $|\Delta_N|\gg N^2$, by making $\epsilon$ smaller we obtain \eqref{eq_cor:Prop4.8}.
\end{proof}

We also need \cite[Lemma 2.5]{Klurman_2021} which we now recall.

\begin{lemma}\label{lem:redConcentration}
    Let $f:\N\to\S^1$ be a multiplicative function such that $f\sim\chi\cdot n^{it}$ for some $t\in\R$ and Dirichlet character $\chi$ of conductor $q$. 
    Let $Q,K\geq 1$ be integers satisfying $\prod_{p\leq K}p\mid Q$ and $q\mid Q$. 
    Then, for any $1\leq a_0\leq Q$ with $(a_0,Q)=1$, we have
    $$\sum_{n\leq N}\Big|f\big(Qn+a_0\big)-\chi(a_0)\cdot (Qn)^{it}\cdot \exp\big(F_Q(f;N)\big)\Big|\ll N\Big(\mathbb{D}\big(f,\chi\cdot n^{it}; K,N\big) + K^{-1/2}\Big),$$
    where
    \begin{equation*}
        F_Q(f;N):= \sum_{p\leq N,\,p\nmid Q} \frac{1}{p}\left(f(p)\cdot \overline{\chi(p)}p^{-it}-1\right).
    \end{equation*}
\end{lemma}

Again, we only make use of this lemma in the form of the next corollary.
\begin{corollary}\label{cor:redConcentration}
    Suppose $f\in\M$ is such that $f\sim\chi\cdot n^{it}$ for some $t\in\R$ and Dirichlet character $\chi$. Let $\epsilon>0$. For any $l_1,l_2\in\Z$ not both zero and large enough $K$, any $Q\in\Phi_K$ satisfies
    \begin{equation*}
        \limsup_{N\to\infty}\Eh \big|f(Q(l_1m+l_2n)+1)-(Q(l_1m+l_2n))^{it}\cdot\exp(F(f;K,N))\big|\leq \epsilon,
    \end{equation*}
    where
    \begin{equation}\label{eq:F(f;K,N)def}
        F(f;K,N):= \sum_{K<p\leq N}\frac{1}{p}\left(f(p)\cdot\overline{\chi(p)}p^{-it}-1\right).
    \end{equation}
\end{corollary}
\begin{proof}
    Note that for large enough $K$ any $Q\in\Phi_K$ satisfies the conditions of \cref{lem:redConcentration}. Hence apply it with $a_0=1$. In particular, $\chi(a_0)=1$. Since $\mathbb{D}(f, \chi\cdot n^{it})<\infty$, we have
    $$\lim_{K\to\infty}\mathbb{D}\big(f, \chi\cdot n^{it}; K, \infty\big) + K^{-1/2} = 0.$$
    Moreover, when $Q\in\Phi_K$, note that $F_Q(f;N) = F(f;K,N)$, which does not depend on $Q$. Thus for large enough $K$ we obtain
    \begin{equation*}
        \limsup_{N\to\infty}\E_{n\in[N]} \big|f(Qn+1)-(Qn)^{it}\cdot\exp(F(f;K,N))\big|\leq \epsilon.
    \end{equation*}
    Since for any $n\leq N$ we have $\big|(Qn)^{it}\cdot \exp\big(F(f;K,N)\big)\big|\leq1$, we may use \cite[Lemma 3.1]{frantzikinakis2024partitionregularitygeneralizedpythagorean} to deduce 
    \begin{equation*}
        \limsup_{N\to\infty}\E_{m,n\in[N]} \big|f(Q(l_1m+l_2n)+1)-(Q(l_1m+l_2n))^{it}\cdot\exp(F(f;K,N))\big|\leq \epsilon,
    \end{equation*}
    Since $|\Delta_N|\gg N^2$, this concludes the proof.
\end{proof}

\subsection{Aperiodic functions}

While concentration estimates are used to handle pretentious functions, in order to handle aperiodic functions we make use of vanishing results that are known for certain 2-variable expressions.
This line of results goes back to the work of Frantzikinakis and Host in \cite{frantzikinakis2016higherorderfourieranalysis} who showed that every aperiodic multiplicative function is Gowers uniform of every order, which in turn implies that statistics involving linear terms in two variables vanish.
More recently, this was extended to allow one arbitrary irreducible quadratic term.

\begin{theorem}[{\cite[Theorem 3.1]{frantzikinakis2024partitionregularitygeneralizedpythagorean}}]\label{thm:aperiodicvanishing}
    Suppose for $s\in\N$ that the multiplicative functions $f_1,\dots,f_s,g\in\M$ are extended to even sequences in $\Z$ and that $f_1$ is aperiodic. Let $L_1,\dots,L_s$ be linear forms in two variables with integer coefficients and suppose that either $s=1$ and $L_1$ is non-trivial, or $s\geq 2$ and the forms $L_1,L_j$ are linearly independent for $j=2,\dots,s$. 
    Let also
    $$P(m,n):= \alpha m^2+\beta mn+ \gamma n^2\in\Z[m,n]$$
    be irreducible and let $Q\in\N$, $\ell,\ell'\in\N_0$. 
    Finally, let $K$ be a convex subset of $\R_+^2$ that is homogeneous, i.e. $(x,y)\in K$ implies $(kx,ky)\in K$ for every $k\in\N$. Then
    $$\lim_{N\to\infty}\E_{m,n\in[N]}\mathbf{1}_K(Qm+\ell,Qn+\ell')\cdot \prod_{j=1}^sf_j\big(L_j(Qm+\ell,Qn+\ell')\big)\cdot g\big(P(Qm+\ell,Qn+\ell')\big)=0.$$
\end{theorem}

We only make use of \cref{thm:aperiodicvanishing} in the special case described by the following corollary.
\begin{corollary}\label{cor:aperiodicJP}
    Let $f,g,h\in\M$, at least one of them aperiodic. 
    Let $a,b,c\in\N$ be perfect squares such that either $a=c$ or $a+b=c$, and let $P_x,P_y,P_z$ be as in \eqref{eq:a=c=1_param} or \eqref{eq:a+b=c_param}, respectively. 
    Then for every $Q\in\N$,
    $$\lim_{N\to\infty}\Eh \Es f\Big(kP_x^Q(m,n)\Big)\cdot g\Big(kP_y^Q(m,n)\Big)\cdot h\Big(kP_z^Q(m,n)\Big)=0$$
\end{corollary}
\begin{proof}
If $h$ is the only aperiodic function, then the product $fgh$ is not the identity, and hence the average in $k$ vanishes by \cref{lem:folnervanish}.

    Otherwise, either $f$ or $g$ are aperiodic, and since $P_x$ and $P_y$ factor into linear terms, the conclusion follows from \cref{thm:aperiodicvanishing} with $K=\{(x,y)\in\R^2:x>(a2+b)y>0\}$.
\end{proof}

\section{Auxiliary results}

\subsection{Structural lemma}\label{subsec:structuralresult}
Let $a,b,c\in\N$ be perfect squares such that either $a=c$ or $a+b=c$.

The following lemma will allow us to reduce \cref{thm_jpforabc} to the case where the involved functions have additional properties. 
\begin{lemma}\label{lem:subgroupGenJP}
    Let $f_1,\dots,f_d\in\M$ be JP w.r.t. $a,b,c$. 
    Denote by $\langle f_1,\dots,f_d\rangle$ the subgroup of $\M$ (under pointwise multiplication) generated by them. 
    Then any $g_1,\dots,g_\ell\in\langle f_1,\dots,f_d\rangle$ are JP w.r.t. $a,b,c$. 
\end{lemma}
\begin{proof}
    Our goal is to show that for every continuous function $\psi:\S^1\to[0,1]$ with $\psi(1)>0$,
    \begin{equation}\label{eq_proof_lem:subgroupGenJP0}
        \limsup_{N\to\infty}\Eh \Es\prod_{j=1}^\ell \psi\Big(g_j\big(kP_x(m,n)\big)\Big)\cdot \psi\Big(g_j\big(kP_y(m,n)\big)\Big)
        \cdot \psi\Big(g_j\big(kP_z(m,n)\big)\Big) >0.
    \end{equation}
    For each $j\in\{1,\dots,\ell\}$ we have
    $$g_j = \prod_{s=1}^d f_s^{w_{s,j}}\text{, for some }w_{s,j}\in\Z,$$
    The coefficients $w_{s,j}$ induce a group homomorphism $T:(\S^1)^{3d}\to(\S^1)^{3\ell}$ given by
    $$T\big((z_{s,i})_{s\leq d,i\leq 3}\big)=\Bigg(\prod_{s=1}^dz_{s,1}^{w_{s,j}},\prod_{s=1}^dz_{s,2}^{w_{s,j}},\prod_{s=1}^dz_{s,3}^{w_{s,j}}\Bigg)_{j\leq\ell}.$$
    Observe that
    \begin{equation}\label{eq_proof_lem:subgroupGenJP1}
        T\big(f_s(kP_w(m,n))_{s\leq d,w\in\{x,y,z\}}\big)= \Big(g_j(kP_x(m,n)),g_j(kP_y(m,n)),g_j(kP_z(m,n))\Big)_{j\leq\ell}.
    \end{equation}
    Define $\phi:(\S^1)^{3\ell}\to[0,1]$ as
    \begin{equation}\label{eq_proof_lem:subgroupGenJP2}
        \phi\big((z_{j,i})_{j\leq\ell, i\leq 3}\big)=\prod_{j=1}^\ell \psi(z_{j,1})\psi(z_{j,2})\psi(z_{j,3}).
    \end{equation}
    By continuity of $T$, $\phi\circ T$ is a continuous function on $(\S^1)^{3d}$ satisfying $(\phi\circ T)(1,\dots,1)>0$. 
    Therefore there exists a continuous function $\tilde\psi:\S^1\to[0,1]$
    such that $\tilde\psi(1)>0$ and\footnote{Here and in the rest of the article, given a function $f:X\to\C$ and $n\in\N$ we denote by $f^{\otimes n}:X^n\to\C$ the product function $f^{\otimes n}(x_1,\dots,x_n)=f(x_1)\cdots f(x_n)$.} $\tilde\psi^{\otimes 3d}\leq\phi\circ T$.
    
    Using the assumption that $f_1,\dots,f_d\in\M$ are JP, it follows that
    \begin{equation*}
        \limsup_{N\to\infty}\Eh\Es \prod_{s=1}^d \tilde\psi\Big(f_s\big(kP_x(m,n)\big)\Big)\cdot \tilde\psi\Big(f_s\big(kP_y(m,n)\big)\Big)
        \cdot \tilde\psi\Big(f_s\big(kP_z(m,n)\big)\Big) >0.
    \end{equation*}
    Combining this with \eqref{eq_proof_lem:subgroupGenJP1} and \eqref{eq_proof_lem:subgroupGenJP2} and the inequality $\tilde\psi^{\otimes 3d}\leq\phi\circ T$, we obtain \eqref{eq_proof_lem:subgroupGenJP0}, finishing the proof.
\end{proof}

The next lemma describes a useful class of generating sets for any given tuple in $\M$. When combined with \cref{lem:subgroupGenJP} it reduces \cref{thm_jpforabc} to a  more manageable special case.

Throughout the rest of the paper we use the following convenient abuse of notation: whenever $f=(f_1,\dots,f_d)\in\M^d$ and $s=(s_1,\dots,s_d)\in\Z^d$ we denote $f_1^{s_1}\dots f_d^{s_d}$ by $f^s$. 
Similarly, when $u=(u_1,\dots,u_d)\in\C^d$ we define $u^s:=u_1^{s_1}\dots u_d^{s_d}$.

\begin{lemma}\label{lem:structuralResult}
    Let $F\subset\M$ be a finite set. Then there exist $g_1,\dots,g_d,h_1,\dots,h_d\in\M$ such that 
    \begin{enumerate}
        \item $F\subset\langle g_1,\dots,g_d,h_1,\dots,h_d\rangle$;
        \item For each $j\in\{1,\dots,d\}$, $h_j\sim n^{it_j}$ for some $t_j\in\R$; 
        \item For every $s\in\Z^d$, the function $g^s$ is either aperiodic or a Dirichlet character.
    \end{enumerate}
\end{lemma}

\begin{proof}
    Let $F=\{f_1,\dots,f_\ell\}$.
    Since the product of two pretentious functions is pretentious, the set $P:=\{s\in\Z^\ell:f^s\text{ is pretentious}\}$ is a subgroup of $\Z^\ell$. Similarly, the set ${\mathcal A}:=\big\{f\in\M:\exists t\in\R: f\sim n^{it}\big\}$ is also a subgroup.

    For each $s\in P$ there exists a unique Dirichlet character $\chi_s$ such that $f^s\overline{\chi_s}\in{\mathcal A}$.
    Define the map $\phi:P\to\mathcal{A}$ such that
    $$\phi(s)=f^s\overline{\chi_s}\text{ so that }f^s=\phi(s)\cdot\chi_s.$$
    Since for each $r,s\in P$ we have $f^rf^s=f^{r+s}=\phi(r+s)\cdot\chi_r\cdot\chi_s$, it follows that $\phi$ is a homomorphism.

    The subgroup of archimedean characters and the subgroup of completely multiplicative functions that pretend to be the trivial function $\mathbf{1}$ are both divisible\footnote{A group $G$ is called \emph{divisible} if for every $g\in G$ and every $n\in\N$ there exists $h\in G$ such that $h^n=g$.}. Hence the abelian subgroup $\mathcal{A}$ is divisible. Equivalently, $\mathcal{A}$ is injective in the category of Abelian groups (see \cite[Ch. 2.3]{weibel1994introduction}). Thus $\phi$ extends to a homomorphism $\tilde{\phi}:\Z^\ell\to{\mathcal A}$.

    We now let $h_i=\tilde{\phi}(e_i)$, where $e_i\in\Z^\ell$ is the $i$-th element of the canonical basis, and $g_i=f_ih_i^{-1}$.
    The first two conditions hold automatically. To verify the third condition, observe that $g^s=f^s\cdot\big(\tilde\phi(s)\big)^{-1}$ is a Dirichlet character when $s\in P$ and it is aperiodic otherwise.
\end{proof}

\subsection{Density lemma}\label{subsec:densityLemma}

The goal of this section is to specify a region $W\subset\N^2$ with positive additive density and with the property, roughly speaking, that given any $(m,n)\in W$ the points $\big(P_x(m,n)\big)^{it},\big(P_y(m,n)\big)^{it},\big(P_z(m,n)\big)^{it}$ are all very close to each other. This will be useful to handle Archimedean characters in the proof of \cref{thm_jpforabc}.

It is only in this section that we use the notion of logarithmic averages,
$$\E_{n\in[N]}^{\log}\phi(n):= \frac{1}{\log N}\sum_{n\in[N]}\frac1n\phi(n),$$
where $N\in\N$ and $\phi: [N]\to\C$. 
Given a vector $v=(v_1,\dots,v_d)\in\T^d$ or $v\in\R^d$, we denote by
$$\|v\|_{\T^d}:=\max_{1\leq i\leq d}\|v_i\|_\T,$$
where for $x\in\R$ we denote by $\|x\|_\T:=\min_{n\in\Z}|x-n|$ the distance to its closest integer.

As usual, for $x\in\T$ we denote by $e(x):=e^{2\pi i x}$.
\begin{lemma}\label{lemma_logweyl}
    Let $d\in\N$ and $(x_n)_{n\in\N}$ be a sequence in $\T^d$.
    Suppose that for every $v\in\Z^d$ the logarithmic averages of the inner product
    $$\E^{\log}_{n\in[N]}e\big(\langle v,x_n\rangle\big)$$
    tend to either $0$ or $1$ when $N\to\infty$.
    Then for every $\epsilon>0$ there exists (a positive logarithmic density set of) $n\in\N$ such that $\|x_n\|_{\T^d}<\epsilon$.
\end{lemma}
\begin{proof}
    For each $N\in\N$ define the probability measures
    $$\mu_N:=\E_{n\in[N]}^{\log}\delta_{x_n}.$$
    Their Fourier transforms are
    $$\widehat{\mu_N}(v) = \int_{\T^d} e\big(\langle v,x\rangle\big)\d\mu_N =\E_{n\in[N]}^{\log}\int_{\T^d}e\big(\langle v,x\rangle\big)\d\delta_{x_n} = \E_{n\in[N]}^{\log}e\big(\langle v,x_n\rangle\big).$$

    Let $H:=\{x\in\T^d:\chi_v(x)=1,\,\forall v\in\Z^d\quad\widehat{\mu_N}(v)\to 1\}$ and note that $H$ is a compact subgroup of $\T^d$. In view of Weyl's criterion (see, e.g. \cite[Proposition 2.4]{Bergelson_2020}) and the hypothesis, it follows that $\mu_N\to\mu$ where $\mu$ is the Haar measure on $H$. 

    Let $\psi:\T^d\to[0,1]$ be a continuous function with $\psi(\mathbf{0})=1$ and $\psi(x)=0$ whenever $\|x\|_{\T^d}>\epsilon$.
    Since $\mathbf{0}\in H$, $\int\psi\d\mu>0$ and from $\mu_N\to\mu$, we conclude that
    $$\E_{n\in[N]}^{\log}\psi(x_n)>0$$
    for all sufficiently large $N$. 
    In particular, $\|x_n\|_{\T^d}<\epsilon$ for some (indeed, a positive logarithmic density set of) $n\in\N$.
\end{proof}

\begin{lemma}\label{lemma:logarithmicuniformdistribution}
    Let $p_1,p_2,p_3\in\Z[x]$
    satisfy
    \begin{equation}\label{eq:conditiontointersectdiagonal0}
       \forall v_1,v_2\in\R,\qquad \lim_{m\to\infty}|p_1(m)|^{v_1}|p_2(m)|^{v_2}|p_3(m)|^{-v_1-v_2}\in\{0,1,\infty\}.
    \end{equation}
    Then for every $\epsilon>0$ and $t_1,\dots,t_d\in\R$ there exist infinitely many $m\in\N$ such that 
    \begin{equation}\label{eq_nearlogspolynomials}
        \forall j\in\{1,\dots,d\}\qquad\forall \ell\in\{1,2\}\qquad \big\|t_j\log |p_\ell(m)|-t_j\log |p_3(m)|\big\|_\T<\epsilon.
    \end{equation}
\end{lemma}
\begin{proof}
    Consider the sequence 
    $$\phi:m\mapsto \left(t_1\log\left|\frac{p_1(m)}{p_3(m)}\right|,t_1\log\left|\frac{p_2(m)}{p_3(m)}\right|,\dots,t_d\log\left|\frac{p_1(m)}{p_3(m)}\right|,t_d\log\left|\frac{p_2(m)}{p_3(m)}\right|\right)\in \T^{2d}.$$
    Note that \eqref{eq_nearlogspolynomials} is equivalent to $\|\phi(m)\|_{\T^{2d}}<\epsilon$. 
    In view of \cref{lemma_logweyl} it suffices to prove that, for every $u\in\Z^{2d}$,
    $$\lim_{N\to\infty}\E^{\log}_{m\in[N]}e\big(\langle u,\phi(m)\rangle\big)\in\{0,1\}.$$
    The inner product $\langle u,\phi(m)\rangle$ is of the form $\log|p_1(m)|^{v_1}|p_2(m)|^{v_2}|p_3(m)|^{-v_1-v_2}$ for some $v_1,v_2\in\R$. We now consider 2 cases

    Case 1: $v_1\deg(p_1)+v_2\deg(p_2)=(v_1+v_2)\deg(p_3)$. 
    
    In this case $\lim_{m\to\infty}|p_1(m)|^{v_1}|p_2(m)|^{v_2}|p_3(m)|^{-v_1-v_2}\in(0,\infty)$, and hence \eqref{eq:conditiontointersectdiagonal0} implies that the limit must be $1$. 
    Therefore $\langle u,\phi(m)\rangle\to0$ as $m\to\infty$ and hence $\lim_{N\to\infty}\E^{\log}_{m\in[N]}e\big(\langle u,\phi(m)\rangle\big)=1.$

    Case 2: $v_1\deg(p_1)+v_2\deg(p_2)\neq(v_1+v_2)\deg(p_3)$.
    
    Then $|\langle u,\phi(m)\rangle-v_3\log(\eta m)|\to0$ for $\eta=\lim_{m\to\infty}\frac1{m^{v_3}}|p_1(m)|^{v_1}|p_2(m)|^{v_2}|p_3(m)|^{-v_1-v_2}\neq0$.
    Since the sequence $m\mapsto v_3\log(\eta m)$ is uniformly distributed modulo $1$ with respect to logarithmic averages, we conclude in this case that
    $$\lim_{N\to\infty}\E^{\log}_{m\in[N]}e(\langle u,\phi(m)\rangle)=0.$$
\end{proof}

\begin{lemma}\label{lem:weightforpolynomials}
    Fix $\epsilon>0$, $t_1,\dots,t_d\in\R$ and let $p_1,p_2,p_3\in\Z[m,n]$ be homogeneous with degree $2$.
    Assume that 
    \begin{equation}\label{eq:conditiontointersectdiagonalold}
        \forall v_1,v_2\in\R,\qquad\lim_{m\to\infty}p_1(m,1)^{v_1}p_2(m,1)^{v_2}p_3(m,1)^{-v_1-v_2}\in\{0,1,\pm\infty\}.
    \end{equation}
    Then for any $0<\theta<1$, the set
    \begin{equation}\label{eq_definitionAm}
        A_m:=\bigcap_{j=1}^d\Big\{n\leq\theta m:\big|p_3(m,n)^{it_j}-p_2(m,n)^{it_j}\big|+\big|p_3(m,n)^{it_j}-p_1(m,n)^{it_j}\big|<\epsilon\Big\}
    \end{equation}
    satisfies $\liminf_{m\to\infty}|A_m|/m>0$.
\end{lemma}

\begin{proof}
    Fix $\epsilon>0$. 
    Choose $\eta>0$ small enough such that whenever $\|u\|_\T<\eta$ we have $|e(u)-1|<\epsilon/2$.

    Let $I\subset(1,\infty)$ be the set of $x>1/\theta$ satisfying
    $$\forall j\in\{1,\dots,d\}\quad\forall\ell\in\{1,2\}\quad \Big\|\frac{t_j}{2\pi}\big(\log \big|p_\ell(x,1)\big|- \log \big|p_3(x,1)\big|\big)\Big\|_\T<\eta.$$
    Note that $I$ is open. 
    By \cref{lemma:logarithmicuniformdistribution}, $I\neq\emptyset$.
    Therefore, the set $J:=\{1/x:x\in I\}$ is a non-empty open subset of $(0,1)$.
    
    For each $m\in\N$ define $B_m:=\{n\in[m]:n/m\in J\}$. 
    By homogeneity of the polynomials $p_1,p_2,p_3$ whenever $n\in B_m$ we have
    $$\forall j\in\{1,\dots,d\}\quad \forall\ell\in\{1,2\}\quad |p_\ell(m,n)^{it_j}-p_3(m,n)^{it_j}|<\epsilon/2.$$
    We conclude that $B_m\subset A_m$, and the proof is completed by the observation that
    $$\lim_{m\to\infty}\frac{|B_m|}m=|J|>0.$$
\end{proof}

\begin{corollary}\label{cor:weightforpolynomials}
    Fix $\epsilon>0$ and $t_1,\dots,t_d\in\R$. 
    Let $a,b,c\in\Z$ be perfect squares such that $a+b=c$ or $a=c$ and let $P_x,P_y,P_z\in\Z[m,n]$ be given by \eqref{eq:a=c=1_param} or \eqref{eq:a+b=c_param}, respectively.
    Define
    $$W:=\bigcap_{j=1}^d\Big\{(m,n)\in\N^2:\big|P_z(m,n)^{it_j}-P_y(m,n)^{it_j}\big|+\big|P_z(m,n)^{it_j}-P_x(m,n)^{it_j}\big|<\epsilon\Big\}.$$
    Then
    $$\liminf_{N\to\infty}\Eh\mathbf{1}_W(m,n)>0.$$
\end{corollary}

\begin{proof}
    We seek to apply \cref{lem:weightforpolynomials} to $P_x,P_y,P_z$, so we first check that for any $v_1,v_2\in\Z$
    $$P(m,1): = P_x(m,1)^{v_1}P_y(m,1)^{v_2}P_z(m,1)^{-v_1-v_2}$$
    converges to either $0,1$ or $\pm\infty$ when $m\to\infty$.

    Case 1: If $a=c$, then
    $$P(m,1)= \frac{(m^2-b)^{v_1}(2\sqrt{a}m)^{v_2}}{(m^2+b)^{v_1+v_2}}= \frac{(1-\frac{b}{m^2})^{v_1}}{\big(1+\frac{b}{m^2}\big)^{v_1+v_2}}\Big(\frac{2\sqrt{a}}{m}\Big)^{v_2}.$$
    Depending on whether $v_2$ is zero, negative or positive, the expression tends to $1,+\infty$ or $0$, respectively.

    Case 2: If $a+b=c$, then
    \begin{align*}
        P(m,1)&= \frac{(m^2-2bm-ab)^{v_1}}{(m^2+2am-ab)^{v_2}}{(m^2+ab)^{v_1+v_2}}=
        \frac{\big(1-2\frac{b}{m}-\frac{ab}{m^2}\big)^{v_1}\big(1+2\frac{a}{m}-\frac{ab}{m^2}\big)^{v_2}}{\big(1+\frac{ab}{m^2}\big)^{v_1+v_2}}
    \end{align*}
    converges to $1$ as $m\to\infty$.
    Applying \cref{lem:weightforpolynomials} with $\theta=1/(a+2b)$, and letting $A_m$ be the set given by \eqref{eq_definitionAm}, we conclude that
    \begin{eqnarray*}
        \liminf_{N\to\infty}\Eh\mathbf{1}_W(m,n)
        &=&
        \liminf_{N\to\infty}\E_{n,m\in[N]}1_{m>(a+2b)n}1_W(m,n)
        \\&=&
        \liminf_{N\to\infty}\E_{m\in[N]}|A_m|/N>0.
    \end{eqnarray*}
\end{proof}

\subsection{A sequential lemma}
The following simple lemma will be used to handle the components of the completely multiplicative functions that pretend to be the constant $1$ function. 

\begin{lemma}\label{lem:sequencesandseries}
    For any $d\in\N$ and any $x:\N\mapsto\R^d$,
    \begin{equation}\label{eq:sequentiallemma}
        \lim_{K\to\infty}\ \liminf_{N\to\infty}\Bigg\|\sum_{K\leq n\leq N}x(n)\Bigg\|_{\T^d}=0.
    \end{equation}
\end{lemma}
Note that the version of this lemma where the $\liminf$ is replaced with a $\limsup$ is false, as seen by taking $d=1$ and $x(n)=1/n$. 
This is the reason it is not possible to replace $\limsup$ with $\liminf$ in the definition of JP.

\begin{proof}
    Let $\pi:\R^d\to\T^d=\R^d/\Z^d$ be the canonical projection and define, for each $N\in\N$, 
    $$X_N:=\pi\left(\sum_{n\leq N}x(n)\right)\in\T^d.$$
    Observe that \eqref{eq:sequentiallemma} can be written as
    $$\lim_{K\to\infty}\liminf_{N\to\infty}\|X_N-X_K\|_{\T^d}=0.$$
    
    Let $S\subset\T^d$ denote all limit points of the sequence $(X_N)_{N=1}^{\infty}$.
    By compactness, for every $\epsilon>0$ there exists $K_0\in\N$ such that for any $K>K_0$ there exists some point $s\in S$ with $\|X_K-s\|_{\T^d}<\epsilon$.
    Take $K>K_0$ and $s\in S$ satisfying $\|X_K-s\|_{\T^d}<\epsilon$, then find an increasing sequence $(N_\ell)_{\ell=1}^\infty$ such that $X_{N_\ell}\to s$ as $\ell\to\infty$.
    It follows that 
    $$\liminf_{N\to\infty}\|X_N-X_K\|_{\T^d}\leq \lim_{\ell\to\infty}\|X_{N_\ell}-X_K\|_{\T^d}\leq
    \lim_{\ell\to\infty}\|X_{N_\ell}-s\|_{\T^d}+\|s-X_K\|_{\T^d}<\epsilon$$
    finishing the proof.
\end{proof}

\begin{corollary}\label{cor:sequencesandseries}
    Let $d\in\N$ and $x:\N\to\C^d$, and suppose that
    $$\sum_{n=1}^\infty\mathfrak{R}(x(n))$$
    converges in $\R^d$. Then
    $$\lim_{K\to\infty}\liminf_{N\to\infty}\max_{1\leq j\leq r}\left|\exp\left(\sum_{K\leq n<N}x_j(n)\right)-1\right|=0.$$
\end{corollary}
\begin{proof}
    Write
    $$S_{K,N}=\sum_{K\leq n<N}x(n).$$
    Convergence of the real part gives
    $$a_K:= \sup_{N>K}\|\mathfrak{R}(S_{K,N})\|_\infty\longrightarrow 0.$$
    Applying \cref{lem:sequencesandseries} to the real sequence $\mathfrak{I}(x(n)/(2\pi))$, we have
    $$\lim_{K\to\infty}\liminf_{N\to\infty}\left\|\frac{\mathfrak{I}(S_{K,N})}{2\pi}\right\|_{\T^d}=0.$$
    The result follows from the triangle inequality,
    $$\max_j\left|e^{S_{K,N}}-1\right|\leq e^{a_k}-1+2\pi\left\|\frac{\mathfrak{I}(S_{K,N})}{2\pi}\right\|_{\T^d}.$$
\end{proof}

\section{Proof of the main result}

\subsection{A sufficient condition for JP}

Let $m,n,Q\in\N$ and $f\in\M$. Given $P_x,P_y,P_z\in\Z[m,n]$, denote
\begin{equation}\label{eq:ABCdef}
    \begin{gathered}
    A_Q^f(m,n):= f\big(P_z^Q(m,n)\big)\overline{f\big(P_x^Q(m,n)\big)},\qquad\qquad B_Q^f(m,n):= f\big(P_z^Q(m,n)\big)\overline{f\big(P_y^Q(m,n)\big)},\\
    C_Q^{f,g}(m,n):= f\big(P_z^Q(m,n)\big)\overline{g\big(P_z^Q(m,n)\big)}.
    \end{gathered}
\end{equation}

\begin{lemma}\label{lem:ABClemma}
    Let $f_1,\dots,f_d\in\M$ and $A^{f_j}_Q(m,n),B^{f_j}_Q(m,n),C^{f_j,f_{j+1}}_Q(m,n)$ be given by \eqref{eq:ABCdef}. If for every continuous $\psi:\S^1\to[0,1]$ with $\psi(1)>0$ there exists $Q\in\N$ such that
    $$\limsup_{N\to\infty}\Eh \Bigg(\prod_{j=1}^d\psi\Big(A^{f_j}_Q(m,n)\Big)\psi\Big(B^{f_j}_Q(m,n)\Big) \prod_{j=1}^{d-1}\psi\Big(C^{f_j,f_{j+1}}_Q(m,n)\Big)\Bigg)>0,$$
    then $f_1,\dots,f_d$ are JP w.r.t. $a,b,c$.
\end{lemma}
\begin{proof}
    Fix a continuous $\psi:\S^1\to[0,1]$ with $\psi(1)>0$ and choose $\epsilon$ small enough so that any point $z\in\S^1$ with $\|z\|_\T<4\epsilon$ satisfies $\psi(z)>\epsilon$.
    
    By hypothesis, there is set $\Lambda\subset\{(m,n)\in\N^2:m>(a+2b)n\}$ with positive upper density such that for any $(m,n)\in\Lambda$ and all $j$,
    $$\|A_Q^{f_j}(m,n)\|_\T<\epsilon,\ \|B_Q^{f_j}(m,n)\|_\T<\epsilon,\ \|C_Q^{f_j,f_{j+1}}(m,n)\|_\T<\epsilon.$$ 
    Therefore, when $(m,n)\in\Lambda$, all $3d$ points
    $$f_j(P_x^Q(m,n)),\ f_j(P_y^Q(m,n)),\ f_j(P_z^Q(m,n)), \quad j=1,\dots,d$$
    lie within an arc $I\subset\S^1$ of length $2\epsilon$.

    In particular, there exists $k\in\N$ such that $f_j(k)\in I$ for all $j$. Since the closure 
    $$\overline{\big\{\big(f_1(k),\dots,f_d(k)\big):k\in\N\big\}}\subset(\S^1)^d$$ is a subgroup, it follows that the set $\Gamma:=\{k\in\N:\forall j\in\{1,\dots,d\},\  f_j(k)\in I^{-1}\}$ is non-empty. 
    Since $\Gamma$ is a multiplicative Bohr set, it has positive multiplicative density.
    Note that for any $k\in\Gamma$ and $(m,n)\in\Lambda$ we have 
    $\|f_j\big(kP_x^Q(m,n)\big)\|_\T<2\epsilon$ for every $j=1,\dots,d$. Similarly with $P_y$ and $P_z$ in place of $P_x$.
    We deduce that for any $(m,n)\in\Lambda$
    $$\Psi(m,n):=\Es\prod_{j=1}^d \psi\Big(f_j\big(kP_x^Q(m,n)\big)\Big)\cdot\psi\Big(f_j\big(kP_y^Q(m,n)\big)\Big)\cdot\psi\Big(f_j\big(kP_z^Q(m,n)\big)\Big)\geq d_\Phi(\Gamma)\epsilon^{3d}.$$ 
    Finally, note that the left-hand side of \eqref{eq_jpdef} is bounded from below by
    $$\limsup_{N\to\infty}\Eh\frac1{Q^2}\mathbf{1}_{\Lambda}(m,n)\Psi(m,n)
    \geq
    \frac1{Q^2}d(\Lambda)d_\Phi(\Gamma)\epsilon^{3d}>0$$
    which shows that $f_1,\dots,f_d$ are JP w.r.t. $P_x$, $P_y$, $P_z$ and completes the proof.
\end{proof}

We shall use the following standard consequence of Fejér's theorem which follows from applying the one-dimensional version coordinate-wise; see \cite[Ch. I, Sec. 2--3]{katznelson2004introduction}. The Fejér kernel on $\T$ is defined as
$$F_r(x):= \sum_{n=-r}^{r}\Big(1-\frac{|n|}{r+1}\Big)e^{2\pi inx},$$
and its Fourier transform is given by
$$\widehat{F_r}(n)=\begin{cases}
    1-\frac{|n|}{r+1},&|n|\leq r,\\
    0,&|n|>r.
\end{cases}$$

\begin{lemma}[Multidimensional Fejér theorem]\label{lem:multidimFejer}
    Let $\rho\in C(\T^l)$. For $r\in\N$, define
    $$F_r^{(l)}(x_1,\dots,x_l)=\prod_{j=1}^l F_r(x_j),$$
    where $F_r$ is the one-dimensional Fejér kernel. 
    Then the convolutions $\rho*F_r^{(l)}$ converge uniformly to $\rho$ as $r\to\infty$.
\end{lemma}

\subsection{Case \texorpdfstring{$a=c$}{a=c} for pretentious functions}
In this section we establish the case of \cref{thm_jpforabc} where $a=c$ and all the functions are pretentious. 
In fact we show that in this case, the functions satisfy the stronger condition in \cref{lem:ABClemma}.
The proof consists of expanding $\psi$ into a Fourier series and using the results collected above to estimate the contribution of different functions.
\begin{theorem}\label{prop:a=c=1PretentiousCase}
    Let $a,b\in\N$ be perfect squares, let $d\in\N$ and let $P_x,P_y,P_z$ be given by \eqref{eq:a=c=1_param}, namely
    $$P_x=m^2-bn^2,\qquad P_y=2\sqrt{a}mn,\qquad P_z=m^2+bn^2.$$
    Let $f_1,\dots, f_d\in\M$ be pretentious. Then $f_1,\dots,f_d$ are JP w.r.t. $a,b,a$.    
\end{theorem}

\begin{proof}
    Suppose that $f_j\sim \chi_j\cdot n^{it_j}$ for some Dirichlet character $\chi_j$ and $t_j\in\R$. Let $\psi:\S^1\to[0,1]$ be continuous with $\psi(1)>0$.
    By \cref{lem:ABClemma}, it suffices to show
    \begin{equation}\label{eq:a=c=1_PretentiousProofGoal}
        \mathcal{L}\left(\prod_{j=1}^d\psi\left(A^{f_j}_Q(m,n)\right)\psi\left(B^{f_j}_Q(m,n)\right) \prod_{j=1}^{d-1}\psi\left(C^{f_j,f_{j+1}}_Q(m,n)\right)\right)>0,
    \end{equation}
    where $A^{f_j}_Q$, $B^{f_j}_Q$ and $C^{f_j,f_{j+1}}_Q$ are given by \eqref{eq:ABCdef} and $\mathcal{L}$ is short for
    \begin{equation}\label{eq:defL}
        \limsup_{K\to\infty} \limsup_{N\to\infty} \E_{Q\in\Phi_K}\Eh.
    \end{equation}
    With a view towards applying the concentration estimates, we let
    \begin{equation*}
        \begin{gathered}
        \tilde{A}^{f_j}(m,n):= \left(\frac{m^2+bn^2}{m^2-bn^2}\right)^{it_j}\exp{\Big(G_{P_z}(f_j;K,N)-2F(f_j;K,N)\Big)},\\
        \tilde{B}^{f_j}_Q(m,n):= \left(\frac{Qm^2+bQn^2}{m}\right)^{it_j}\overline{f_j(2\sqrt{a}Qn)}\exp\Big(G_{P_z}(f_j;K,N)-F(f_j;K,N)\Big),\\
        \tilde{C}^{f_j,f_{j+1}}_Q(m,n):= \Big((Qm)^2+b(Qn)^2\Big)^{i(t_j-t_{j+1})}\exp\Big(G_{P_z}(f_j;K,N)-G_{P_z}(f_{j+1};K,N)\Big),
        \end{gathered}
    \end{equation*}
    where $G_P$ is defined in \eqref{eq:G(f;K,N)def} and $F$ is defined in \eqref{eq:F(f;K,N)def}.
    Write
    \begin{equation*}
        \begin{gathered}
        \alpha_j^Q(m,n):= f_j(P_z^Q(m,n)),\quad \tilde{\alpha}_j^Q(m,n):= P_z(Qm,Qn)^{it_j}\exp(G_{P_z}(f_j;K,N)),\\
        \beta_j^Q(m,n):= f_j(P_x^Q(m,n)),\quad \tilde{\beta}_j^Q(m,n):= P_x(Qm,Qn)^{it_j}\exp(2F(f_j;K,N)).
        \end{gathered}
    \end{equation*}
    Since $P_z(1,0)=1$ and $P_z$ is irreducible, by \cref{cor:irrConcentration}, for every $\epsilon>0$, once $K$ is big enough, every $Q\in\Phi_K$ satisfies
    \begin{equation}\label{eq:irrEstimate}
        \limsup_{N\to\infty}\Eh \big|\alpha_j^Q(m,n)-\tilde{\alpha}_j^Q(m,n)\big|\leq \epsilon.
    \end{equation}
    Since $P_x$ is reducible, we can factor $P_x^Q(m,n)= L_1^Q(m,n)\cdot L_2^Q(m,n)$, with $L_1^Q(m,n):=Q\big(m+\sqrt bn\big)+1$ and $L_2^Q(m,n):=Q\big(m-\sqrt bn\big)+1$. 
    Since $f_j$ is completely multiplicative, we have
    $$\beta_j^Q(m,n)= f_j(L_1^Q(m,n))\cdot f_j(L_2^Q(m,n)).$$
    Applying \cref{cor:redConcentration} to each linear factor separately, for every $\epsilon>0$, once $K$ is large enough, every $Q\in\Phi_K$ satisfies
    \begin{equation}\label{eq:redEstimate}
        \limsup_{N\to\infty}\Eh\big|\beta_j^Q(m,n)-\tilde{\beta}_j^Q(m,n)\big|\leq \epsilon.
    \end{equation}
    Note that $\tilde{A}^{f_j}(m,n)=\tilde{\alpha}_j^Q(m,n)\cdot\overline{\tilde{\beta}_j^Q(m,n)}$ and that
    $$\big|A_Q^{f_j}-\tilde{A}^{f_j}\big|= \big|\alpha_j^Q\overline{\beta_j^Q}-\tilde{\alpha}_j^Q\overline{\tilde{\beta}_j^Q}\big|\leq \big|\alpha_j^Q-\tilde{\alpha}_j^Q\big|\cdot\big|\beta_j^Q\big|+ \big|\tilde{\alpha}_j^Q\big|\cdot\big|\beta_j^Q-\tilde{\beta}_j^Q\big|.$$
    Observe that $\big|\beta_j^Q\big|=1$ and that $\big|\tilde{\alpha}_j^Q\big|$ is bounded for large enough $K$. Thus, averaging, taking the limits and using \eqref{eq:irrEstimate} and \eqref{eq:redEstimate}, we find that
    $$\mathcal{L}\Big|A^{f_j}_Q(m,n)-\tilde{A}_Q^{f_j}(m,n)\Big|=0.$$
    Then, since $\S^1$ is compact, $\psi$ is uniformly continuous and so
    \begin{equation}\label{eq:ATildeEstimate}
        \mathcal{L}\Big|\psi\Big(A^{f_j}_Q(m,n)\Big)-\psi\Big(\tilde{A}_Q^{f_j}(m,n)\Big)\Big|=0.
    \end{equation}
    Similar arguments can be used to show that also
    \begin{equation}\label{eq:BCTildeEstimate}
        \mathcal{L}\Big|\psi\Big(B_Q^{f_j}(m,n)\Big)-\psi\Big(\tilde{B}^{f_j}(m,n)\Big)\Big|=0,\quad \mathcal{L}\Big|\psi\Big(C_Q^{f_j,f_{j+1}}(m,n)\Big)-\psi\Big(\tilde{C}^{f_j,f_{j+1}}_Q(m,n)\Big)\Big|=0.
    \end{equation}
    
    Inserting the estimates \eqref{eq:ATildeEstimate} and \eqref{eq:BCTildeEstimate} into \eqref{eq:a=c=1_PretentiousProofGoal}, it now suffices to prove
    \begin{equation}\label{eq:a=c=1_PretentiousProofGoal2}
        \mathcal{L}\left(\prod_{j=1}^d\psi\left(\tilde A^{f_j}(m,n)\right)\psi\left(\tilde B_Q^{f_j}(m,n)\right) \prod_{j=1}^{d-1}\psi\left(\tilde C^{f_j,f_{j+1}}_Q(m,n)\right)\right)>0.
    \end{equation}
    We establish \eqref{eq:a=c=1_PretentiousProofGoal2} by expanding $\psi$ into its Fourier series.
    Certain terms in the series will vanish when we take the average over $Q$, so we collect those frequencies that survive this step in the group $Z\subset\Z^{2d-1}$ described by
    $$Z=\left\{s=(s_1,s_2)\in\Z^d\times\Z^{d-1}: f^{s_1}(n)\equiv n^{i\big(\sum_{j=1}^ds_{1,j}t_j+2\sum_{j=1}^{d-1}s_{2,j}(t_j-t_{j+1})\big)}\right\}.$$
    It is now convenient to introduce the notation
    $$\tilde{B}_Q^f(m,n):= \Big(\tilde{B}_Q^{f_1}(m,n),\dots,\tilde{B}_Q^{f_d}(m,n)\Big),\quad \tilde C^{f}_Q(m,n):=\Big(\tilde C^{f_1,f_2}_Q(m,n),\dots,\tilde C^{f_{d-1},f_d}_Q(m,n)\Big).$$
    Recall the notation $\psi^{\otimes^{2d-1}}:(\S^1)^{2d-1}\to[0,1]$ to denote the product function
    $$(z_1,\dots,z_{2d-1})\mapsto\psi(z_1)\cdots\psi(z_{2d-1}),$$
    and the notation $z^c:=z_1^{c_1}\dots z_r^{c_r}$ for $c=(c_1,\dots,c_r)\in\Z^r$ and $z=(z_1,\dots,z_r)\in(\S^1)^r$.
    The projection of $\psi^{\otimes^{2d-1}}$ to the frequencies in $Z$ is the function $\rho:(\S^1)^{2d-1}\to[0,1]$ defined by
    $$\rho(z)=\int_{Z^\perp}\big(\psi\otimes\cdots\otimes\psi\big)(zz')d z',$$
    where $Z^\perp:=\{z\in (\S^1)^{2d-1}:(\forall s\in Z),\ z^s=1\}$ is the dual of the group $Z$. 
    Hence the Fourier transform of $\rho$ is
    \begin{equation*}
        \widehat{\rho}(u,v):=\int_{(\S^1)^{2d-1}}\rho(z)z^{-(u,v)}\d z=
    \begin{cases}
        \widehat{\psi^{\otimes^{2d-1}}}(u,v),&\text{if }(u,v)\in Z\\
        0,&\text{otherwise}.
    \end{cases}
    \end{equation*}
    
    On the way to establish \eqref{eq:a=c=1_PretentiousProofGoal2} we first show that
    \begin{equation*}
        \begin{gathered}
        \mathcal{L}\Bigg(\prod_{j=1}^d\psi\Big(\tilde A^{f_j}(m,n)\Big)\psi\Big(\tilde B_Q^{f_j}(m,n)\Big) \prod_{j=1}^{d-1}\psi\left(\tilde C^{f_j,f_{j+1}}_Q(m,n)\right)\Bigg)\nonumber\\
        \geq
        \mathcal{L}\Bigg(\prod_{j=1}^d\psi\Big(\tilde A^{f_j}(m,n)\Big)\rho\Big(\tilde B_Q^f(m,n),\tilde C^{f}_Q(m,n)\Big)\Bigg).
        \end{gathered}
    \end{equation*}
    Since $\mathcal{L}$ is sub-additive, writing $\gamma:= \rho-\psi^{\otimes^{2d-1}}$, it suffices to prove that
    \begin{equation}\label{eq:a=c=1_PretentiousProofGoal4}
        \mathcal{L}\Bigg(\prod_{j=1}^d\psi\left(\tilde A^{f_j}(m,n)\right)\gamma\Big(\tilde B^{f}_Q(m,n),\tilde C^{f}_Q(m,n)\Big)\Bigg)=0.
    \end{equation}
    Let $\mathcal{L}^*$ be short for
    \begin{equation}\label{eq:defL*}
        \lim_{K\to\infty} \lim_{N\to\infty} \E_{Q\in\Phi_K}\Eh,
    \end{equation}
    which we use only when the limits involved all exist.
    \begin{proof}[Proof of \eqref{eq:a=c=1_PretentiousProofGoal4}]
    \renewcommand\qedsymbol{$\blacktriangle$}
    Note that $\gamma$ only has frequencies in $\Z^{2d-1}\setminus{Z}$, i.e., $\hat\gamma(s)=0$ whenever $s\in Z$. 
    Applying \cref{lem:multidimFejer} with $l=2d-1$, the convolutions $\gamma\ast F_r$, where $F_r$ are the Fej\'er kernels, are trigonometric polynomials in $(\S^1)^{2d-1}$ that approximate $\gamma$ in the supremum norm as $r\to\infty$. 
    It then suffices to prove that for every $r\in\N$,
    $$\mathcal{L}^*\Bigg(\prod_{j=1}^d\psi\Big(\tilde A^{f_j}(m,n)\Big)\big(\gamma\ast F_r\big)\Big(\tilde B_Q^f(m,n),\tilde C^{f}_Q(m,n)\Big)\Bigg)=0.$$
    As each $\gamma\ast F_r$ is a finite linear combination of powers $z^s$ with $s\in\Z^{2d-1}\setminus Z$ and ${\mathcal L}^*$ is linear, it suffices to show that, for any $s=(s_1,s_2)\in\Z^{2d-1}\setminus Z$,
    \begin{equation}\label{eq:a=c=1_PretentiousProofGoal4.1}
        \mathcal{L}^*\Bigg(\prod_{j=1}^d\psi\Big(\tilde A^{f_j}(m,n)\Big)\cdot\Big(\tilde B_Q^f(m,n)\Big)^{s_1}\Big(\tilde C^{f}_Q(m,n)\Big)^{s_2}\Bigg)=0.
    \end{equation}
    Next note that $\tilde B_Q^{f_j}(m,n)=Q^{it_j}\overline{f_j(Q)}\times($terms that do not depend on $Q)$, and similarly $\tilde C^{f_j,f_{j+1}}_Q(m,n)=Q^{2i(t_j-t_{j+1})}\times($terms that do not depend on $Q)$.
    When $s\notin Z$ we deduce that the product $\Big(\tilde B_Q^f(m,n)\Big)^{s_1}\Big(\tilde C^{f}_Q(m,n)\Big)^{s_2}$ contains a non-trivial function of $Q$ as a factor.
    This implies, via \cref{lem:folnervanish}, that the average over $Q$ vanishes, and hence establishes \eqref{eq:a=c=1_PretentiousProofGoal4.1}.
    \end{proof}

    In view of \eqref{eq:a=c=1_PretentiousProofGoal4}, we have reduced \eqref{eq:a=c=1_PretentiousProofGoal2} to the following
    \begin{equation}\label{eq:a=c=1_PretentiousProofGoal5}
        \mathcal{L}\Bigg(\prod_{j=1}^d\psi\left(\tilde A^{f_j}(m,n)\right)\rho\Big(\tilde B_Q^f(m,n),\tilde C^{f}_Q(m,n)\Big)\Bigg)>0,
    \end{equation}
    Roughly speaking, we have reduced matters to those frequencies in $Z$. As far as these frequencies are concerned, $f^s$ behaves as an Archimedean character, so we will be able to replace $\tilde{B}_Q^f(m,n)$ and $\tilde C^{f}_Q(m,n)$ with simpler expressions. 
    To this end, let
    \begin{align*}
        \dbtilde{B}^f(m,n):= \Bigg(&\left(\frac{m^2+bn^2}{2\sqrt{a}mn}\right)^{it_1}\exp\big(G_{P_z}(f_1;K,N)-F(f_1;K,N)\big),\dots,\\
        &\left(\frac{m^2+bn^2}{2\sqrt{a}mn}\right)^{it_d}\exp\big(G_{P_z}(f_d;K,N)-F(f_d;K,N)\big)\Bigg),
    \end{align*}
    $$\dbtilde{C}^f(K,N):= \Big(\exp\big(G_{P_z}(f_1;K,N)-G_{P_z}(f_2;K,N)),\dots,\exp\big(G_{P_z}(f_{d-1};K,N)-G_{P_z}(f_d;K,N)\big)\Big).$$
    The next step is to show that
    \begin{equation}\label{eq:a=c=1_PretentiousProofGoal5.1}
        \mathcal{L}^*\Bigg(\prod_{j=1}^d\psi\left(\tilde A^{f_j}(m,n)\right)\bigg(\rho\Big(\tilde{B}_Q^f(m,n),\tilde C^{f}_Q(m,n)\Big)-\rho\Big(\dbtilde{B}^f(m,n),\dbtilde{C}^f(K,N)\Big)\bigg)\Bigg)=0.
    \end{equation}
    \begin{proof}[Proof of \eqref{eq:a=c=1_PretentiousProofGoal5.1}]
    \renewcommand\qedsymbol{$\blacktriangle$}
    Recall that $\hat\rho(s)=0$ unless $s\in Z$. 
    Applying \cref{lem:multidimFejer} with $l=2d-1$, it suffices to prove \eqref{eq:a=c=1_PretentiousProofGoal5.1} when $\rho$ is replaced by powers $z^s$ with $s=(s_1,s_2)\in Z$. 
    Since $s\in Z$, it follows that $f^{s_1}$ is an Archimedean character.
    On the other hand, (from the definition of $t_j$) we know that $f^{s_1}$ pretends to be the product of $n^{i\sum_{j=1}s_{1,j}t_j}$ and a Dirichlet character.
    This can only happen if $f^{s_1}(n)=n^{i\sum_{j=1}s_{1,j}t_j}$ and hence it follows that $\sum_{j=1}^{d-1}s_{2,j}(t_j-t_{j+1})=0$. 
    It now follows directly that for every $Q,m,n$ we have
    $$\Big(\tilde{B}_Q^f(m,n)\Big)^{s_1}\Big(\tilde C^{f}_Q(m,n)\Big)^{s_2} = \Big(\dbtilde{B}^f(m,n)\Big)^{s_1}\Big(\dbtilde C^{f}(K,N)\Big)^{s_2}.$$
    \end{proof}

    In view of \eqref{eq:a=c=1_PretentiousProofGoal5.1} we have reduced \eqref{eq:a=c=1_PretentiousProofGoal5} to
    \begin{equation}\label{eq:a=c=1_PretentiousProofGoal6}
        \mathcal{L}\Bigg(\prod_{j=1}^d\psi\left(\tilde A^{f_j}(m,n)\right)\rho\Big(\dbtilde B^{f}(m,n),\dbtilde{C}^{f}(K,N)\Big)\Bigg)>0.
    \end{equation}

    To prove \eqref{eq:a=c=1_PretentiousProofGoal6} we will first get rid of the Archimedean characters involved in the expression, making use of \cref{cor:weightforpolynomials}. 
    To this end, let  
    \begin{equation*}
        \begin{gathered}
        \dbtilde{A}^{f_j}(K,N):= \exp\Big(G_{P_z}(f_j;K,N)-2F(f_j;K,N)\Big),\\
        \check{B}^f(K,N):= \Big(\exp\big(G_{P_z}(f_1;K,N)-F(f_1;K,N)\big),\dots,\exp\big(G_{P_z}(f_d;K,N)-F(f_d;K,N)\big)\Big).\\
        \end{gathered}
    \end{equation*}
    Let $\delta>0$ be a small parameter to be chosen later (depending only on $f_1,\dots,f_d$ and $\psi$), then let $\epsilon>0$ be sufficiently small depending on $\delta$ and let $W\subset\N^2$ be the set given by \cref{cor:weightforpolynomials} such that whenever $(m,n)\in W$ we have
    \begin{equation}\label{eq:a=c=1archimedianestimate}
    \big|(m^2+bn^2)^{it_j}-(m^2-bn^2)^{it_j}\big|+ \big|(m^2+bn^2)^{it_j}-(2\sqrt{a}mn)^{it_j}\big|<\epsilon
    \end{equation}
    for all $j\in\{1,\dots,d\}$. 
    Since $\rho$ and $\psi$ are continuous functions on compact sets, they are uniformly continuous. 
    Hence, if $\epsilon$ is sufficiently small and using \eqref{eq:a=c=1archimedianestimate}, 
    it follows that for any $(m,n)\in W$ and $j\in\{1,\dots,d\}$,
    $$\left|\psi\left(\tilde A^{f_j}(m,n)\right)-\psi\left(\dbtilde{A}^{f_j}(K,N)\right) \right|<\frac\delta{2d}.$$
    A similar estimate can be used to replace $\dbtilde B$ with $\check{B}$, and hence we can then estimate \eqref{eq:a=c=1_PretentiousProofGoal6} from below by
    \begin{equation}\label{eq:a=c=1_PretentiousProofGoal7}
        \underline{d}(W)\cdot\mathcal{L}\Bigg(\prod_{j=1}^d\psi\Big(\dbtilde{A}^{f_j}(K,N)\Big)\rho\Big(\check{B}^f(K,N),\dbtilde{C}^f(K,N)\Big)\Bigg) -\delta.
    \end{equation}
    Note that the expression inside $\mathcal{L}$ in \eqref{eq:a=c=1_PretentiousProofGoal7} does not depend on $m,n$ or $Q$. Therefore we can replace $\mathcal{L}$ with $\limsup_{K\to\infty}\limsup_{N\to\infty}$.
    Each of $F$ and $G_{P_z}$ are complex sums of the form $\sum_{K\leq p\leq N}x(p)$. Pretentiousness gives that the real parts of $F$ and $G_{P_z}$ converge to $0$. Thus we can invoke \cref{cor:sequencesandseries} for $\C^{2d-1}$ (and using the fact that both $\psi$ and $\rho$ are continuous) to deduce that \eqref{eq:a=c=1_PretentiousProofGoal7} is at least
    $$\underline{d}(W)\cdot\psi(1)^d\cdot\rho(1,\dots,1)-\delta>0$$
    for sufficiently small $\delta$. This establishes \eqref{eq:a=c=1_PretentiousProofGoal6} and thus finishes the proof.
\end{proof}

\subsection{Case \texorpdfstring{$a+b=c$}{a+b=c} for pretentious functions}
In this section we establish the case of \cref{thm_jpforabc} where $a+b=c$ and all the functions are pretentious. The proof is similar to that of \cref{prop:a=c=1PretentiousCase} with some crucial differences.
To highlight these differences, we omit several of the details that were already explained in the proof of \cref{prop:a=c=1PretentiousCase}.

\begin{theorem}\label{prop:a+b=cPretentiousCase}
    Let $a,b\in\N$ be perfect squares, let $d\in\N$ and let $P_x,P_y,P_z$ be given by \eqref{eq:a+b=c_param}, namely
    \begin{equation*}
    \begin{gathered}
        P_x=P_x(m,n)=m^2-2bmn-abn^2,\qquad\qquad\qquad\qquad P_y=P_y(m,n)=m^2+2amn-abn^2,\\
        P_z=P_z(m,n)=m^2+abn^2.
    \end{gathered}
    \end{equation*}
    Let $f_1,\dots,f_d\in\M$ be pretentious. Then $f_1,\dots,f_d$ are JP w.r.t. $P_x,P_y,P_z$.
\end{theorem}

\begin{proof}
 Suppose that $f_j\sim\chi_j\cdot n^{it_j}$ for some Dirichlet character $\chi_j$ and $t_j\in\R$.
    Let $\psi:\S^1\to[0,1]$ be continuous with $\psi(1)>0$. By \cref{lem:ABClemma}, it suffices to prove
    \begin{equation}\label{eq:a+b=c_PretentiousProofGoal}
    \mathcal{L} \Bigg(\prod_{j=1}^d\psi\Big(A^{f_j}_Q(m,n)\Big)\psi\Big(B^{f_j}_Q(m,n)\Big) \prod_{j=1}^{d-1}\psi\Big(C^{f_j,f_{j+1}}_Q(m,n)\Big)\Bigg)>0,
    \end{equation}
    where $A_Q^{f_j},B_Q^{f_j},C_Q^{f_j,f_{j+1}}$ are given by \eqref{eq:ABCdef} and $\mathcal{L}$ by \eqref{eq:defL}

    With a view towards applying the concentration estimates, we let
    $$\tilde{A}^{f_j}(m,n):= \left(\frac{m^2+abn^2}{m^2-2bmn-abn^2}\right)^{it_j}\exp{\Big(G_{P_z}(f_j;K,N)-2F(f_j;K,N)\Big)},$$
    $$\tilde{B}^{f_j}(m,n):= \left(\frac{m^2+abn^2}{m^2+2amn-abn^2}\right)^{it_j}\exp{\Big(G_{P_z}(f_j;K,N)-2F(f_j;K,N)\Big)},$$
    $$\tilde{C}^{f_j,f_{j+1}}_Q(m,n):= \big((Qm)^2+ab(Qn)^2\big)^{i(t_j-t_{j+1})}\exp\big(G_{P_z}(f_j;K,N)-G_{P_z}(f_{j+1};K,N)\big).$$
    Arguing as in the previous proof, in view of \cref{cor:irrConcentration} and \cref{cor:redConcentration} we have
    \begin{equation*}
        \begin{gathered}
            \mathcal{L}\Big|\psi\big(A_Q^{f_j}(m,n)\big)-\psi\big(\tilde{A}^{f_j}(m,n)\big)\Big|=0, \qquad\mathcal{L}\Big|\psi\big(B_Q^{f_j}(m,n)\big)-\psi\big(\tilde{B}^{f_j}(m,n)\big)\Big|=0,\\
            \mathcal{L}\Big|\psi\big(C_Q^{f_j,f_{j+1}}(m,n)\big)-\psi\big(\tilde{C}^{f_j,f_{j+1}}_Q(m,n)\big)\Big|=0.
        \end{gathered}
    \end{equation*}
    Inserting these estimates into \eqref{eq:a+b=c_PretentiousProofGoal}, we reduce it to
    \begin{equation}\label{eq:a+b=c_PretentiousProofGoal2}
        \mathcal{L}\Bigg(\prod_{j=1}^d\psi\Big(\tilde A^{f_j}(m,n)\Big)\psi\Big(\tilde B^{f_j}(m,n)\Big) \prod_{j=1}^{d-1}\psi\Big(\tilde C^{f_j,f_{j+1}}_Q(m,n)\Big)\Bigg)>0.
    \end{equation}

    To write more compactly we will denote
    $$\tilde C^{f}_Q(m,n):= \Big(\tilde C^{f_1,f_2}_Q(m,n),\dots,\tilde C^{f_{d-1},f_d}_Q(m,n)\Big).$$
    We consider the auxiliary subgroup $Z\subset\Z^{d-1}$ and the function $\rho:(\S^1)^{d-1}\to[0,1]$ defined by
    $$Z=\Bigg\{s\in\Z^{d-1}: {\sum_{j=1}^{d-1}s_j(t_j-t_{j+1})}=0\Bigg\},\quad\text{ and }\quad\rho(z)=\int_{Z^\perp}\psi^{\otimes^{d-1}}(zz')d z',$$
    where $Z^\perp:=\big\{z\in (\S^1)^{d-1}:(\forall s\in Z),\ z^s=1\big\}$ is the dual of the group $Z$. 
    The Fourier transform of $\rho$ is
    \begin{equation*}
        \widehat{\rho}(u,v)=
    \begin{cases}
        \widehat{\psi^{\otimes^{d-1}}}(u,v),&\text{if }(u,v)\in Z\\
        0,&\text{otherwise}.
    \end{cases}
    \end{equation*}
    
    On the way to establish \eqref{eq:a+b=c_PretentiousProofGoal2}, we first show that
    \begin{equation}\label{eq:a+b=c_PretentiousProofGoal4}
        \mathcal{L}\Bigg(\prod_{j=1}^d\psi\Big(\tilde A^{f_j}(m,n)\Big)\psi\Big(\tilde B^{f_j}(m,n)\Big)\gamma\Big(\tilde C^{f}_Q(m,n)\Big)\Bigg)=0,
    \end{equation}
    where  $\gamma:= \rho-\psi^{\otimes^{d-1}}$. The proof of \eqref{eq:a+b=c_PretentiousProofGoal4} is identical to the proof of \eqref{eq:a=c=1_PretentiousProofGoal4} above, with the only difference being that here Fejér's theorem is applied in $(\S^1)^{d-1}$, and so we have omitted the details.
    
    In view of \eqref{eq:a+b=c_PretentiousProofGoal4}, we have reduced \eqref{eq:a+b=c_PretentiousProofGoal2} to
    \begin{equation}\label{eq:a+b=c_PretentiousProofGoal5}
        \mathcal{L}\Bigg(\prod_{j=1}^d\psi\Big(\tilde A^{f_j}(m,n)\Big)\psi\Big(\tilde B^{f_j}(m,n)\Big)\rho\Big(\tilde C^{f}_Q(m,n)\Big)\Bigg)>0.
    \end{equation}
    Having reduced matters to frequencies in $Z$, we now want to replace $\tilde{C}_Q^{f}(K,N)$ with a simpler expression. To that end, let
    $$\dbtilde{C}^{f}(K,N):= \Big(\exp\big(G_{P_z}(f_1;K,N)-G_{P_z}(f_2;K,N)\big),\dots,\exp\big(G_{P_z}(f_{d-1};K,N)-G_{P_z}(f_d;K,N)\big)\Big).$$

    Let $\mathcal{L}^*$ be given by \eqref{eq:defL*}. The first step to obtain \eqref{eq:a+b=c_PretentiousProofGoal5} is to show that
    \begin{equation}\label{eq:a+b=c_PretentiousProofGoal5.1}
        \mathcal{L}^*\Bigg(\prod_{j=1}^d\psi\Big(\tilde A^{f_j}(m,n)\Big)\psi\Big(\tilde B^{f_j}(m,n)\Big)\cdot\Big(\rho\Big(\tilde C^{f}_Q(m,n)\Big)-\rho\Big(\dbtilde{C}^{f}(K,N)\Big)\Big)\Bigg)=0.
    \end{equation}

    Indeed, since $\hat{\rho}(s)=0$ unless $s\in Z$, after applying \cref{lem:multidimFejer} with $l=d-1$ it suffices to establish \eqref{eq:a+b=c_PretentiousProofGoal5.1} when $\rho$ is replaced by powers $z^s$ with $s\in Z$. 
    But for any $s\in Z$, by definition we have $\Big(\tilde C^{f}_Q(m,n)\Big)^s = \Big(\dbtilde C^{f}(K,N)\Big)^s,$
    which confirms \eqref{eq:a+b=c_PretentiousProofGoal5.1}.
    
    In view of \eqref{eq:a+b=c_PretentiousProofGoal5.1} we have reduced \eqref{eq:a+b=c_PretentiousProofGoal5} to
    \begin{equation}\label{eq:a+b=c_PretentiousProofGoal6}
        \mathcal{L}\Bigg(\prod_{j=1}^d\psi\big(\tilde A^{f_j}(m,n)\Big)\psi\Big(\tilde B^{f_j}(m,n)\Big)\rho\Big(\dbtilde{C}^{f}_Q(K,N)\Big)\Bigg)>0.
    \end{equation}

    Let $\epsilon>0$ be a small parameter to be chosen later (depending only on $f_1,\dots,f_d$ and $\psi$) and let $W\subset\N^2$ be the set given by \cref{cor:weightforpolynomials} such that whenever $(m,n)\in W$ we have
    \begin{equation}\label{eq:archimedianestimate}
    \big|(m^2+abn^2)^{it_j}-(m^2-2bmn-abn^2)^{it_j}\big|+ \big|(m^2+abn^2)^{it_j}-(m^2+2amn-abn^2)^{it_j}\big|<\epsilon
    \end{equation}
    for all $j\in\{1,\dots,d\}$. 
    Hence, for every $\delta>0$, choosing $\epsilon>0$ sufficiently small and using \eqref{eq:archimedianestimate} we can then estimate \eqref{eq:a+b=c_PretentiousProofGoal6} from below by
    \begin{equation}\label{eq:a+b=c_PretentiousProofGoal7}
        \underline{d}(W)\cdot\mathcal{L}\Bigg(\prod_{j=1}^d\psi\Big(\dbtilde{A}^{f_j}(K,N)\Big)\psi\Big(\dbtilde{B}^{f_j}(K,N)\Big)\rho\Big(\dbtilde{C}^f(K,N)\Big)\Bigg) -\delta,
    \end{equation}
    where $\underline{d}(W)>0$ by \cref{cor:weightforpolynomials} and
    \begin{equation*}
        \begin{gathered}
            \dbtilde{A}^{f_j}(K,N):= \exp{\Big(G_{P_z}(f_j;K,N)-2F(f_j;K,N)\Big)},\\
            \dbtilde{B}^{f_j}(K,N):= \exp{\Big(G_{P_z}(f_j;K,N)-2F(f_j;K,N)\Big)}.
        \end{gathered}
    \end{equation*}
    Invoking \cref{cor:sequencesandseries} for $\C^{2d-1}$ establishes \eqref{eq:a+b=c_PretentiousProofGoal6}.
\end{proof}

\subsection{Case \texorpdfstring{$a=c$}{a=c}}

In this section we prove \cref{thm_jpforabc} in the case $a=c$, for arbitrary functions $f_1,\dots,f_d\in\M$.
To this end, we use the following sufficient criterion for JP, which amounts to consider the average on $(m,n)$ on a grid $(Qm+1,Qn)$, and then average multiplicatively over all $Q$.

\begin{lemma}\label{lem:averageInQReduction}
    Let $f_1,\dots,f_d\in\M$ and $P_x,P_y,P_z$ be given by \eqref{eq:a=c=1_param} or \eqref{eq:a+b=c_param}. If for every continuous $\psi:\S^{1}\to[0,1]$ with $\psi(1)>0$
    \begin{equation}\label{eq_lem:averageInQReduction}
    \lim_{K\to\infty}\limsup_{N\to\infty} \E_{Q\in\Phi_K}\Eh\Es \prod_{j=1}^d\psi\Big(f_j\big(kP_x^Q(m,n)\big)\Big)\cdot \psi\Big(f_j\big(kP_y^Q(m,n)\big)\Big)\cdot \psi\Big(f_j\big(kP_z^Q(m,n)\big)\Big)>0,
    \end{equation}
    then $f_1,\dots,f_d$ are JP w.r.t. $a,b,c$.
\end{lemma}
\begin{proof}
    Since $\limsup$ is sub-additive, the left-hand side in \eqref{eq_lem:averageInQReduction} is smaller or equal than
    $$\limsup_{K\to\infty}\E_{Q\in\Phi_K}\limsup_{N\to\infty}\ \Eh\Es \prod_{j=1}^d\psi\Big(f_j\big(kP_x^Q(m,n)\big)\Big)\cdot \psi\Big(f_j\big(kP_y^Q(m,n)\big)\Big)\cdot \psi\Big(f_j\big(kP_z^Q(m,n)\big)\Big).$$
    Therefore, \eqref{eq_lem:averageInQReduction} implies that there exists $Q$ for which the inner limsup in $N$ is positive.
    Since the grid $\{(Qm+1,Qn):m,n\in\N\}$ has positive density in $\N^2$, the JP property follows.
\end{proof}

We are now ready to prove the \cref{thm_jpforabc} in the case $a=c$.
In view of Lemmas \ref{lem:subgroupGenJP} and \ref{lem:structuralResult}, it suffices to establish the following.
\begin{theorem}\label{thm:general_a=c}
    Let $g_1,\dots,g_d,h_1,\dots,h_d\in\M$ be such that each $h_j(n)\sim n^{it_j}$ and for every $s=(s_1,\dots,s_d)\in\Z^d$, $g^s$ is either aperiodic or a Dirichlet character. 
    Let also $a,b\in\N$ be perfect squares and let $P_x,P_y,P_z$ be given by \eqref{eq:a=c=1_param}, namely
    $$P_x=m^2-bn^2,\qquad P_y=2\sqrt{a}mn,\qquad P_z=m^2+bn^2.$$
    Then $g_1,\dots,g_d,h_1,\dots,h_d$ are JP w.r.t. $a,b,a$.
\end{theorem}

\begin{proof}
    Let $\mathcal{L}$ and $\mathcal{L}^*$ be short for
    $$\lim_{K\to\infty}\limsup_{N\to\infty} \E_{Q\in\Phi_K}\Eh\Es,\quad\text{ and }\quad\lim_{K\to\infty}\lim_{N\to\infty} \E_{Q\in\Phi_K}\Eh\Es,$$
    respectively (where we only use $\mathcal{L}^*$ when the limits involved exist).

    By \cref{lem:averageInQReduction}, our goal is to establish
    $$\mathcal{L} \Bigg(\prod_{j=1}^d\psi\big(g_j(kP_x^Q)\big)\psi\big(g_j(kP_y^Q)\big)\psi\big(g_j(kP_z^Q)\big)\psi\big(h_j(kP_x^Q)\big)\psi\big(h_j(kP_y^Q)\big)\psi\big(h_j(kP_z^Q)\big)\Bigg)>0.$$
    Consider the subgroup $Z\subset(\Z^d)^3$ and the function $\rho:(\S^1)^{3d}\to[0,1]$ defined by
    $$Z:=\big\{s=(s_1,s_2,s_3)\in\Z^{3d}:g^{s_1},g^{s_3}\text{ are Dirichlet characters and }g^{s_2}\equiv g^{s_1+s_2+s_3}\equiv 1\big\},$$
    $$\rho(z)=\int_{Z^\perp}(\psi\otimes\cdots\otimes\psi)(zz')\d z',$$
    where $Z^\perp= \{z\in (\S^1)^{3d}:\forall s\in Z\ z^s=1\}$ is the dual group of $Z$ in $(\S^1)^{3d}$. Then
    $$\widehat{\rho}(s)=
    \begin{cases}
        \widehat{\psi^{\otimes^{3d}}},&\text{if }s\in Z\\
        0,&\text{otherwise}.
    \end{cases}$$
    Since $\psi\geq 0$ and $(1,\dots,1)\in Z^\perp$, it follows that $C:=\rho(1,\dots,1)=\int_{Z^\perp}\psi^{\otimes^{3d}}(z)\d z>0$.
    
    For each $\gamma\in\big\{(\rho-\psi^{\otimes^{3d}})\otimes\psi^{\otimes^{3d}},(\psi^{\otimes^{3d}}-\rho)\otimes\psi^{\otimes^{3d}}\big\}$, we begin by proving that
    \begin{equation}\label{eq:a=c=1ReductionMainGoal2}
        \mathcal{L}^*\Big(\gamma\big(g(kP_x^Q),g(kP_y^Q),g(kP_z^Q),h(kP_x^Q),h(kP_y^Q),h(kP_z^Q)\big)\Big)= 0,
    \end{equation}
    
    \begin{proof}[Proof of \eqref{eq:a=c=1ReductionMainGoal2}]
    \renewcommand\qedsymbol{$\blacktriangle$}    
     Applying \cref{lem:multidimFejer} with $l=6d$, the convolutions $\gamma\ast F_r$, where $F_r$ are the Fejér kernels, are trigonometric polynomials in $(\S^1)^{6d}$ that approximate $\gamma$ in the uniform norm as $r\to\infty$. 
     To show \eqref{eq:a=c=1ReductionMainGoal2} it then suffices to prove
    \begin{equation}\label{eq:a=c=1ReductionMainGoal2.1}
        \forall r\in\N,\quad \mathcal{L}^*\Big(\big(\gamma\ast F_r\big)\big(g(kP_x^Q),g(kP_y^Q),g(kP_z^Q),h(kP_x^Q),h(kP_y^Q),h(kP_z^Q)\big)\Big)= 0.
    \end{equation}
    Note that $\widehat{\gamma\ast F_r}(s)=\hat \gamma(s)\hat{F_r}(s)=0$ whenever $s=(s_1,s_2,s_3,s_4,s_5,s_6)\in Z\times\Z^{3d}$.
    It follows that each convolution $\gamma\ast F_r$ is a finite linear combination of powers $z^s$ with $s\in\Z^{6d}\setminus (Z\times\Z^{3d})$, hence it suffices to show for each such $s=(s_1,s_2,s_3,s_4,s_5,s_6)$ that
    \begin{equation}\label{eq:a=c=1ReductionMainGoal2.2}
        \mathcal{L}^*\Big(\big(g^{s_1}\cdot h^{s_4}\big)\big(kP_x^Q)(g^{s_2}\cdot h^{s_5})(kP_y^Q)(g^{s_3}\cdot h^{s_6})(kP_z^Q)\big)\Big)=0.
    \end{equation}
    As $s\notin Z\times\Z^3$ we have three cases.

    Case 1: The product $g^{s_1+s_2+s_3}\cdot h^{s_4+s_5+s_6}$ is non-trivial. Then \eqref{eq:a=c=1ReductionMainGoal2.2} follows by \cref{lem:folnervanish}.
    
    Case 2: At least one of $g^{s_1}, g^{s_2}, g^{s_3}$ is aperiodic. 
    This implies that at least one of $g^{s_1}\cdot h^{s_4}, g^{s_2}\cdot h^{s_5}, g^{s_3}\cdot h^{s_6}$ is aperiodic and hence by \cref{cor:aperiodicJP} we conclude that \eqref{eq:a=c=1ReductionMainGoal2.2} holds.
    
    Case 3: Neither Case 1 or Case 2 hold and $s\notin Z\times\Z^3$. 
    In this case, all of $g^{s_1}$, $g^{s_2}$ and $g^{s_3}$ are Dirichlet character and $g^{s_1+s_2+s_3}\cdot h^{s_4+s_5+s_6}\equiv1$. 
    In particular, after factoring out the average in $k$, which equals $1$, one may disregard $k$ entirely.
    Since $h^{s_4+s_5+s_6}$ pretends to be an Archimedean character and $g^{s_1+s_2+s_3}$ is a Dirichlet character, this also implies that $g^{s_1+s_2+s_3}\equiv 1$. 
    The condition $s\notin Z\times\Z^3$ now implies that $g^{s_2}\not\equiv1$.
    
    If $K$ is large enough, then for any $Q\in\Phi_K$ and any $m,n\in\N$,
    $$g^{s_1}(P_x^Q)=g^{s_3}(P_z^Q)=1\text{ and }g^{s_2}(P_y^Q)=g^{s_2}(2\sqrt{a}Qn).$$
    In view of Corollaries \ref{cor:irrConcentration} and \ref{cor:redConcentration}, we have
    \begin{equation*}
        \begin{gathered}
            \mathcal{L}^*\Big|h^{s_4}(P_x^Q)- ((Qm)^2-b(Qn)^2)^{its_4}\cdot\exp(2s_4F(h;K,N))\Big| =0,\\
            \mathcal{L}^*\Big|h^{s_5}(P_y^Q)-h^{s_5}(2\sqrt{a}Qn)\cdot (Qm)^{its_5}\cdot\exp(s_5F(h;K,N))\Big| =0,\\
            \mathcal{L}^*\Big|h^{s_6}(P_z^Q)-((Qm)^2+b(Qn)^2)^{its_6}\cdot\exp(s_6G_{P_z}(h;K,N))\Big| =0.
        \end{gathered}
    \end{equation*}
    Using these we replace the left-hand side of \eqref{eq:a=c=1ReductionMainGoal2.2} with
    \begin{equation*}
        \begin{gathered}
        \mathcal{L}^*\Big(g^{s_2}(2\sqrt{a}Qn)\cdot((Qm)^2-b(Qn)^2)^{its_4}\cdot h^{s_5}(2\sqrt{a}Qn)\cdot (Qm)^{its_5}\cdot ((Qm)^2+b(Qn)^2)^{its_6}\cdot\zeta\Big)\\
        =\mathcal{L}^*\Big((g^{s_2}\cdot h^{s_5})(Q) Q^{i(2s_4+s_5+2s_6)}\cdot g^{s_2}(2\sqrt{a}n)\cdot(m^2-bn^2)^{its_4} h^{s_5}(2\sqrt{a}n)\cdot m^{its_5} (m^2+bn^2)^{its_6}\cdot\zeta\Big),
        \end{gathered}
    \end{equation*}
    where $\zeta=\exp\big(2s_4F(h;K,N)+s_5F(h;K,N)+s_6G_{P_z}(h;K,N)\big)$. 
    Since $g^{s_2}$ is a non-trivial Dirichlet character and the map $Q\mapsto h^{s_5}(Q)\cdot Q^{i(2s_4+s_5+2s_6)}$ pretends to be an Archimedean character, we must then have $g^{s_2}(Q)\cdot h^{s_5}(Q)\cdot Q^{i(2s_4+s_5+2s_6)}\not\equiv 1$. 
    Therefore the average vanishes due to \cref{lem:folnervanish}.
    \end{proof}
    
    Next, we show that
     \begin{equation}\label{eq:a=c=1ReductionMainGoal3}
        \begin{gathered}
        \mathcal{L}^* \Big(\Big(\rho\big(g(kP_x^Q),g(kP_y^Q),g(kP_z^Q)\big)-C\Big)\cdot\psi^{\otimes^{3d}}\big(h(kP_x^Q),h(kP_y^Q),h(kP_z^Q)\big)\Big)=0.
        \end{gathered}
    \end{equation}

    \begin{proof}[Proof of \eqref{eq:a=c=1ReductionMainGoal3}]
    \renewcommand\qedsymbol{$\blacktriangle$}
    Recall that $\widehat{\rho}$ is supported in $Z$. 
    Thus, through a similar approximation argument as before, it suffices to establish \eqref{eq:a=c=1ReductionMainGoal3} when $\rho$ is replaced by powers $z^s$ with $s\in Z$, namely
     \begin{equation}
        \begin{gathered}
        \mathcal{L}^* \Big(\Big(g^{s_1}(kP_x^Q)g^{s_2}(kP_y^Q)g^{s_3}(kP_z^Q)-1\Big) \cdot\psi^{\otimes^{3d}}\big(h(kP_x^Q),h(kP_y^Q),h(kP_z^Q)\big)\Big)=0.
        \end{gathered}
    \end{equation}
    In that case, $g^{s_2}\equiv1$ and for any $Q\in\Phi_K$, with large enough $K$, we have $g^{s_1}(P_x^Q)=g^{s_3}(P_z^Q)=1$. 
    
    \end{proof}

    To complete the proof, note that \eqref{eq:a=c=1ReductionMainGoal2} and \eqref{eq:a=c=1ReductionMainGoal3} together imply
    \begin{equation}\label{eq:a=c=1ReductionMainGoal4}
        \begin{gathered}
        \mathcal{L}\Big(\ \psi^{\otimes^{6d}}\big(g(kP_x^Q),g(kP_y^Q),g(kP_z^Q),h(kP_x^Q),h(kP_y^Q),h(kP_z^Q)\big)\Big)\\
        \geq C\cdot\mathcal{L}\Big(\ \psi^{\otimes^{3d}}\big(h(kP_x^Q),h(kP_y^Q),h(kP_z^Q)\big)\Big).
        \end{gathered}
    \end{equation}
    Since $C>0$, using \cref{prop:a=c=1PretentiousCase} this finishes the proof.
\end{proof}

\subsection{Case \texorpdfstring{$a+b=c$}{a+b=c}}

In this section we prove \cref{thm_jpforabc} in the case $a+b=c$, for arbitrary functions $f_1,\dots,f_d\in\M$.
The proof is similar to the case $a=c$ established in the previous section but with a few important differences. In particular, instead of considering an average over $Q$, it is more convenient to simply choose a sufficiently divisible $Q$, depending only on $f_1,\dots,f_d$, as described by the following lemma.
\begin{lemma}\label{lem:findingQ}
    Let $g=(g_1,\dots,g_d)\in{\mathcal M}^d$. 
    Then there exists $Q\in\N$ such that for every $s\in\Z^d$ for which the product $g^s$ is a Dirichlet character, $g^s(Qn+1)=1$ for all $n\in\N$.
\end{lemma}
\begin{proof}
    Let $\mathcal{S}:=\{s\in\Z^d: g^s\text{ is a Dirichlet character}\}$. Note that $\mathcal{S}$ is a subgroup of $\Z^d$, and thus it is finitely generated, say $\mathcal{S}=\langle{s_1,\dots,s_\ell}\rangle$. For each $i\in\{1,\dots, \ell\}$, the function $g^{s_i}$ is a Dirichlet character, thus there exists $q_i\in\N$ such that $g^{s_i}(q_in+1)=1$ for every $n\in\N$. 
    Taking $Q:=\text{lcm}(q_1,\dots,q_\ell)$, it follows that $g^s(Qn+1)=1$ for every $n\in\N$ when $s\in \mathcal{S}=\langle{s_1,\dots,s_\ell}\rangle$.
\end{proof}

Here is the final result needed to establish \cref{thm_jpforabc}.
The main difference in the proof, relatively to the proof of \cref{thm:general_a=c} above is that here we do not need to use a concentration estimate directly. This allows us to rely on \cref{lem:findingQ} to find a suitable $Q$ from the beginning of the proof, considerably simplifying the argument.

\begin{theorem}\label{thm:general_a+b=c}
    Let $g=(g_1,\dots,g_d),h=(h_1,\dots,h_d)\in\M^d$ be such that each $h_j(n)\sim n^{it_j}$ and for every $s=(s_1,\dots,s_d)\in\Z^d$, $g^s$ is either aperiodic or a Dirichlet character. Let also $a,b\in\N$ and $P_x,P_y,P_z$ be given by \eqref{eq:a+b=c_param}, namely
    \begin{equation*}
        \begin{gathered}
        P_x=P_x(m,n)=m^2-2bmn-abn^2,\qquad\qquad\qquad\qquad    P_y=P_y(m,n)=m^2+2amn-abn^2,\\
        P_z=P_z(m,n)=m^2+abn^2.
        \end{gathered}
    \end{equation*}
    Then $g_1,\dots,g_d,h_1,\dots,h_d$ are JP w.r.t. $a,b,a+b$.
\end{theorem}
\begin{proof}
    Let $Q_0\in\N$ be given by \cref{lem:findingQ}, and let $\mathcal{L}$ and $\mathcal{L}^*$ be short for
    $$\limsup_{N\to\infty}\Eh\Es,\quad\text{ and }\quad \lim_{N\to\infty}\Eh\Es.$$ 
    We aim to show that, for every multiple $Q$ of $Q_0$,
    \begin{equation*}
        \mathcal{L} \Bigg(\prod_{j=1}^d\psi\big(g_j(kP_x^Q)\big)\psi\big(g_j(kP_y^Q)\big)\psi\big(g_j(kP_z^Q)\big)\psi\big(h_j(kP_x^Q)\big)\psi\big(h_j(kP_y^Q)\big)\psi\big(h_j(kP_z^Q)\big)\Bigg)>0.
    \end{equation*}
    We make use of a subgroup $Z$ of $(\Z^d)^3$ and a function $\rho:(\S^1)^{3d}\to[0,1]$ defined as
    $$Z:=\big\{s=(s_1,s_2,s_3)\in\Z^{3d}:g^{s_1},g^{s_2},g^{s_3}\text{ are Dirichlet characters and }g^{s_1+s_2+s_3}\equiv 1\big\},$$
    and
    $$\rho(z)=\int_{Z^\perp}(\psi\otimes\cdots\otimes\psi)(zz')\d z',$$
    where $Z^\perp= \{z\in (\S^1)^{3d}:\forall s\in Z\ z^s=1\}$ is the dual of $Z$ in $(\S^1)^{3d}$. 
    Then
    $$\widehat{\rho}(s)=
    \begin{cases}
        \widehat{\psi^{\otimes^{3d}}},&\text{if }s\in Z\\
        0,&\text{otherwise}.
    \end{cases}$$

    We first prove that
    \begin{equation}\label{eq:a+b=cReductionMainGoal2}
        \mathcal{L}^* \Big(\gamma\big(g(kP_x^Q),g(kP_y^Q),g(kP_z^Q),h(kP_x^Q),h(kP_y^Q),h(kP_z^Q)\big)\Big)= 0,
    \end{equation}
    where $\gamma:=\big(\rho-\psi^{\otimes^{3d}}\big)\otimes\psi^{\otimes^{3d}}$.
    
    \begin{proof}[Proof of \eqref{eq:a+b=cReductionMainGoal2}]
    \renewcommand\qedsymbol{$\blacktriangle$}
    Arguing as in the proof of \eqref{eq:a=c=1ReductionMainGoal2}, it suffices to show that for each $s\in\Z^{6d}\setminus{(Z\times\Z^{3d})}$ we have
    $$\mathcal{L}^*\Big(\big(g^{s_1}\cdot h^{s_4}\big)\big(kP_x^Q)(g^{s_2}\cdot h^{s_5})(kP_y^Q)(g^{s_3}\cdot h^{s_6})(kP_z^Q)\big)\Big)=0.$$
    
    Case 1: If at least one of $g^{s_1}, g^{s_2}, g^{s_3}$ is aperiodic, we use \cref{cor:aperiodicJP} to conclude that \eqref{eq:a+b=cReductionMainGoal2} holds.

    Case 2: All of $g^{s_1},g^{s_2},g^{s_3}$ are Dirichlet characters and $g^{s_1+s_2+s_3}\not\equiv1$. Since $h^{s_4+s_5+s_6}$ pretends to be an Archimedean character, we must have $g^{s_1+s_2+s_3}\cdot h^{s_4+s_5+s_6}\not\equiv 1$. The result then follows by \cref{lem:folnervanish}.
    \end{proof}
    
    Let $C:=\rho(1,\dots,1)=\int_{Z^\perp}\psi^{\otimes^{3d}}(z)\d z$.
    Since $\psi\geq0$, $\psi(1)>0$ and $(1,\dots,1)\in Z^\perp$, it follows that $C>0$. 
    We next show that
    \begin{equation}\label{eq:a+b=cReductionMainGoal3.1}
        \begin{gathered}
        \mathcal{L}^* \Big(\Big(\rho\big(g(kP_x^Q),g(kP_y^Q),g(kP_z^Q)\big)-C\Big)\cdot\psi^{\otimes^{3d}}\big(h(kP_x^Q),h(kP_y^Q),h(kP_z^Q)\big)\Big)= 0.
        \end{gathered}
    \end{equation}

    \begin{proof}[Proof of \eqref{eq:a+b=cReductionMainGoal3.1}]
    \renewcommand\qedsymbol{$\blacktriangle$}
    Recall that $\rho$ only has frequencies in $Z$. 
    Again, by \cref{lem:multidimFejer}, it suffices to show \eqref{eq:a+b=cReductionMainGoal3.1} when $\rho$ is replaced by powers $z^s$ with $s\in Z$, that is
    $$\mathcal{L}^* \Big(\Big(g^{s_1}(kP_x^Q)g^{s_2}(kP_y^Q)g^{s_3}(kP_z^Q)-1\Big)\cdot\psi^{\otimes^{3d}}\big(h(kP_x^Q),h(kP_y^Q),h(kP_z^Q)\big)\Big)=0.$$
    In that case, $g^{s_1+s_2+s_3}(k)=1$. Moreover, since $P_x^Q\equiv P_y^Q\equiv P_z^Q\equiv 1\bmod{Q}$ and $Q$ is a multiple of $Q_0$ given by \cref{lem:findingQ}, we have $g^{s_1}(P_x^Q)=g^{s_2}(P_y^Q)=g^{s_3}(P_z^Q)=1$.
    \end{proof}
    
    Note that \eqref{eq:a+b=cReductionMainGoal2} and \eqref{eq:a+b=cReductionMainGoal3.1} together imply
    \begin{align*}
        \begin{gathered}
        \mathcal{L}\Big( \psi^{\otimes^{6d}}\big(g(kP_x^Q),g(kP_y^Q),g(kP_z^Q),h(kP_x^Q),h(kP_y^Q),h(kP_z^Q)\big)\Big)\\
        =C\cdot \mathcal{L}\Big( \psi^{\otimes^{3d}}\big(h(kP_x^Q),h(kP_y^Q),h(kP_z^Q)\big)\Big).
        \end{gathered}
    \end{align*}
    Since $C> 0$, the conclusion now follows from \cref{prop:a+b=cPretentiousCase}.
\end{proof}


\bibliographystyle{plain}
\bibliography{references}

\end{document}